\documentclass[pdflatex]{sn-jnl}

\usepackage[english]{babel}
\usepackage{caption,subcaption,float,blindtext,tabu,mathtools,nicefrac,colonequals,csquotes,xcolor,multicol,longtable}

\usepackage{stmaryrd}
\usepackage[normalem]{ulem}

\renewcommand{\R}{\mathbb{R}}
\renewcommand{\Q}{\mathbb{Q}}
\renewcommand{\Z}{\mathbb{Z}}

\let\int\relax
\DeclareMathOperator{\int}{int}
\DeclareMathOperator{\conv}{conv}

\newcommand{\I}{\mathcal{I}}

\renewcommand{\A}{\mathcal{A}}
\newcommand{\proj}{\textnormal{proj}}
\newcommand{\LB}{\textnormal{LB}}
\newcommand{\UB}{\textnormal{UB}}
\newcommand{\lb}{\textnormal{lb}}

\newcommand{\ALG}{\mathtt{ALG}}

\newcommand{\st}{ :}

\newcommand{\norm}[1]{\left\lVert#1\right\rVert}
\newcommand{\Ball}[3]{B_{#2,#3}(#1)}
\newcommand{\preGrid}{\tilde{\Lambda}^{\textup{Grid}}}
\newcommand{\Grid}{\Lambda^{\textup{Grid}}}
\newcommand{\Compact}{\Lambda^{\textup{compact}}}
\newcommand{\flag}{bool}

\newcommand{\cInd}{\textup{Ind}^{\textup{c}}}

\newcommand{\DYAlg}{\texttt{D\&Y}}
\newcommand{\GridAlg}{\texttt{GRID}}
\newcommand{\FPTOAAlg}{\texttt{OAA}}

\newcommand{\Chris}{\texttt{Christofides}}

\newcommand{\ExtGreedy}{\texttt{Extended Greedy}}

\newcommand{\git}{\url{https://gitlab.rlp.net/shelfric/molib}}

\jyear{2022}%

\theoremstyle{thmstyleone}%
\newtheorem{theorem}{Theorem}[section]
\newtheorem{proposition}[theorem]{Proposition}%
\newtheorem{corollary}[theorem]{Corollary}
\newtheorem{lemma}[theorem]{Lemma}
\newtheorem{observation}[theorem]{Observation}
\newtheorem{property}{Property}

\theoremstyle{thmstyletwo}%

\theoremstyle{thmstylethree}%
\newtheorem{definition}[theorem]{Definition}%
\newtheorem{assumption}[theorem]{Assumption}
\usepackage{natbib}
\setcitestyle{authoryear} 

\begin{document}

\title[Constructing Convex Approximation Sets]{Efficiently Constructing Convex Approximation Sets in Multiobjective Optimization Problems}

\author*[1]{\fnm{Stephan} \sur{Helfrich}}\email{helfrich@mathematik.uni-kl.de}

\author[1]{\fnm{Stefan} \sur{Ruzika}}\email{ruzika@mathematik.uni-kl.de}

\author[2,3]{\fnm{Clemens} \sur{Thielen}}\email{clemens.thielen@tum.de}

\affil*[1]{\orgdiv{Department of Mathematics}, \orgname{RPTU Kaiserslautern-Landau}, \orgaddress{\street{Paul-Ehrlich-Str.~14}, \city{Kaiserslautern}, \postcode{67663}, \country{Germany}}}

\affil[2]{\orgdiv{TUM Campus Straubing for Biotechnology and Sustainability}, \orgname{Weihenstephan-Triesdorf University of Applied Sciences}, \orgaddress{\street{Am~Essigberg~3}, \city{Straubing}, \postcode{94315}, \country{Germany}}}
\affil[3]{\orgdiv{Department of Mathematics, School of Computation, Information and Technology}, \orgname{Technical University of Munich}, \orgaddress{\street{Boltzmannstr.~3}, \city{Garching bei München}, \postcode{85748}, \country{Germany}}}


\abstract{
     Convex approximation sets for multiobjective optimization problems are a well-studied relaxation of the common notion of approximation sets. Instead of approximating each image of a feasible solution by the image of some solution in the approximation set up to a multiplicative factor in each component, a convex approximation set only requires this multiplicative approximation to be achieved by some convex combination of finitely many images of solutions in the set. This makes convex approximation sets efficiently computable for a wide range of multiobjective problems - even for many problems for which (classic) approximations sets are hard to compute.
 
 
 \smallskip
 
 In this article, we propose a polynomial-time algorithm to compute convex approximation sets that builds upon an exact or approximate algorithm for the weighted sum scalarization and is, therefore, applicable to a large variety of multiobjective optimization problems. The provided convex approximation quality is arbitrarily close to the approximation quality of the underlying algorithm for the weighted sum scalarization. In essence, our algorithm can be interpreted as an approximate variant of the dual variant of Benson's Outer Approximation Algorithm. Thus, in contrast to existing convex approximation algorithms from the literature, information on solutions obtained during the approximation process is utilized to significantly reduce both the practical running time and the cardinality of the returned solution sets while still guaranteeing the same worst-case approximation quality. We underpin these advantages by the first comparison of all existing convex approximation algorithms on several instances of the triobjective knapsack problem and the triobjective symmetric metric traveling salesman problem.
 }

\keywords{Multi-Objective Optimization; Approximation Algorithm; Convex Approximation Sets; Benson's Method}


\maketitle

\section{Introduction}
Almost any decision process or challenge constitutes a multiobjective optimization problem, i.e., there are multiple conflicting goals to consider: social benefit may conflict with cost, safety may conflict with personal freedom, and environmental benefit may conflict with profit, to name just a few examples. These problems generally have no unique best solution, and, if no prior information about preferences is available, every \emph{efficient solution} is of interest -- solutions for which each any solution that is better in some goal is necessarily worse in at least one other goal. Hereby, a major challenge is the typically enormous number of images of efficient solutions.

\emph{Approximation} allows to substantially reduce the number of required images while still obtaining a provable  solution quality. Here, it is sufficient to find a set of (not necessarily efficient) solutions, called an \emph{approximation set}, that, for each possible image, contains a solution whose image is component-wise at least as good up to a multiplicative factor. In their seminal work on approximation,~\cite{Papadimitriou+Yannakakis:multicrit-approx} show that, for any~$\varepsilon>0$,  polynomial-sized approximation sets achieving a factor of $1 + \varepsilon$ in every component always exist under weak assumption and can efficiently constructed if an only if an approximate variant of the decision problem associated with the multiobjective optimization problem can efficiently be solved. There exist, however, problem classes for which the construction of such $(1+\varepsilon)$-approximation sets is surmised to be difficult: For example,~\cite{Papadimitriou+Yannakakis:multicrit-approx} show that, unless~$\textsf{P} = \textsf{NP}$, there is no FPTAS for constructing a $(1+\varepsilon)$-approximation set for the biobjective minimum $s$-$t$-cut problem. On the contrary, the single-objective minimum $s$-$t$-cut problem can be solved exactly in polynomial-time.

This strongly motivates to study the~\emph{weighted sum scalarization} in this context, where scalarized single-objective problems are constructed by means of a weighted sum of the objective functions. It is known that, in instances of multiobjective \emph{minimization} problems with $d$~objectives, optimal solution sets for the weighted sum scalarization (i.e., sets of solutions that contain, for each possible weighting of the objectives, an optimal solution for the associated scalarized optimization problem) always constitute approximation sets with approximation factor~$d$~\citep{Bazgan+etal:parametric,Glasser+etal:multi-hardness,Helfrich+etal:approximation-via-scalarization}. However, optimal solution sets for the weighted sum scalarization can be exponentially large as well~\citep{Carstensen:parametric,Gassner+Klinz:parametric-assignment,Nikolova+etal:ESA06,Ruhe:parametric-network-flows}. Moreover, in instances of multiobjective \emph{maximization} problems, 
no constant approximation quality can be obtained in general by means of the weighted sum scalarization~\citep{Bazgan+etal:parametric,Glasser+etal:multi-hardness,Helfrich+etal:approximation-via-scalarization}.

This surprisingly indicates, at first, that the weighted sum scalarization is rather useless for the approximation of multiobjective maximization problems. However, this is not the case anymore when considering the slightly relaxed concept of approximation introduced by~\cite{Diakonikolas+Yannakakis:epsilon-convex}: so-called~\emph{convex approximation sets} contain, for each possible image, finitely many solutions such that a \emph{convex combination of their images} is component-wise as good up to a multiplicative factor. In particular, \cite{Diakonikolas:thesis,Helfrich:mp-approximation} present generic algorithms that, given an instance of a multiobjective minimization or maximization problem, efficiently compute convex approximation sets based on a polynomial-time exact algorithm, an approximation scheme, or an approximation algorithm for the weighted sum scalarization.
As a consequence, $(1 + \varepsilon)$-convex approximation sets (where, for each possible image, some convex combination of finitely many images is as good up to the factor~$1 + \varepsilon$) can be constructed in polynomial time if and only if a polynomial-time exact algorithm or an approximation scheme for the weighted sum scalarization is available.

In this article, we continue this line of research on convex approximation sets. Based on a characterization of convex approximation sets via solution sets that contain an approximate solution for the weighted sum scalarization for each possible weight vector, we present an algorithm for the efficient construction of convex approximation sets that is applicable to a large variety of multiobjective optimization problems and builds upon an exact or approximate algorithm for the weighted sum scalarization. In contrast to the existing algorithms, the algorithm for the weighted sum scalarization is called adaptively, i.e., based on information obtained during the approximation process, to improve the practical running time and reduce the cardinality of the returned solution sets. We underpin these advantages of our algorithms with the first performance study of all existing convex approximation algorithms designed so far. 

\subsection{Related Literature}
A recent survey on exact solution methods for multiobjective optimization problems is provided by~\cite{Halffmann+etal:survey}.
The probably most-applied algorithm to compute optimal solution sets for the weighted sum scalarization for instances of \emph{biobjective} optimization problems is the dichotomic search method~\citep{Aneja1979bicritransport,Cohon2004multiobjectiveprogramming}. In case of more than two objectives, algorithms that explicitly or implicitly determine the so-called weighted sum weight set decomposition~\citep{Alves2016esn,Boekler:output-sensitive,Halffmann2020,Przybylski2010ESN,Oezpeynirci2010ESN} are capable of computing such solution sets. For multiobjective linear programs, Benson-type algorithms~\citep{Benson1998OuterApproximation,Hamel2014BensonType} and their dual variant~\citep{Ehrgott2012DualBenson} can additionally be applied. 

\medskip

For an extensive survey on general approximation methods, which seek to work under very weak assumptions, and approximation methods tailored to multiobjective problems with a particular structure, we refer to~\cite{Herzel+etal:survey}. Almost all general approximation methods for multiobjective optimization problems build upon the seminal work of \cite{Papadimitriou+Yannakakis:multicrit-approx}, who show that polynomial-sized approximation sets exists under weak assumptions. Subsequent work focuses on approximation methods that, given~$\alpha \geq 1$, compute approximation sets whose cardinality is bounded in terms of the cardinality of the smallest possible $\alpha$-approximation set while maintaining or only slightly worsening the approximation quality~$\alpha$~\citep{Bazgan+etal:min-pareto,Diakonikolas+Yannakakis:approx-pareto-sets,Diakonikolas+Yannakakis:epsilon-convex,Koltun+Papadimitriou:approx-dom-repr,Vassilvitskii+Yannakakis:trade-off-curves}. Additionally, the existence result of \cite{Papadimitriou+Yannakakis:multicrit-approx} has recently been improved in \cite{Herzel+etal:dualrestrict} who show that, for any $\varepsilon>0$, a polynomial-sized approximation set that is exact in one objective while ensuring an approximation quality of $1 + \varepsilon$ in all other objectives always exists under the same assumptions.

However, as already outlined, there exists problems for which the results of~\cite{Papadimitriou+Yannakakis:multicrit-approx} and succeeding articles cannot be used, but efficient methods exists for the problems induced by the weighted sum scalarization. Hence, how the weighted sum scalarization can be employed for approximation has naturally been in focus of research as well. The results of~\cite{Glasser+etal:multi-hardness-proceedings,Glasser+etal:multi-hardness} imply that, in each instance of each $p$-objective \emph{minimization} problem and for any $\varepsilon>0$, a $((1 + \varepsilon) \cdot \delta \cdot p)$-approximation set can be computed in fully polynomial time provided that the objective functions are positive-valued and polynomially computable and a $\delta$-approximation algorithm for the optimization problems induced by the weighted sum scalarization exists. 
\cite{Halfmann+etal:general-approx} present a method to obtain, for each instance of each \emph{biobjective minimization} problem and for any $0 < \varepsilon \leq 1$, an approximation set that guarantees an approximation quality of $(\delta \cdot (1 + 2\varepsilon))$ in one objective function while still obtaining an approximation quality of at least $(\delta \cdot (1 + \frac{1}{\varepsilon}))$ in the other objective function, provided a polynomial-time $\delta$-approximation algorithm for the problems induced by the weighted sum scalarization is available. This \enquote{trade-off} between the approximation qualities in the individual objectives is studied in more detailed by~\cite{Bazgan+etal:power-weighted-sum}, who introduce a multi-factor notion of approximation and present a method that, for each $\varepsilon>0$ and in each instance of each $p$-objective \emph{minimization} problem for which a polynomial-time $\delta$-approximation algorithm for the problems induced by the weighted sum scalarization exists, computes a set of solutions such that every feasible solution is component-wise approximated within some vector $(\alpha_1, \ldots, \alpha_p)$ of approximation factors~$\alpha_i \geq 1$ such that $\sum_{i: \alpha_i > 1} \alpha_i = \delta \cdot p + \varepsilon$.

In each of these articles on the weighted sum scalarization, it is shown that the methods and approximation results can \emph{not} be translated to maximization problems in general.
Nevertheless, \cite{Daskalakis+etal:Chord-Algorithm,Diakonikolas+Yannakakis:epsilon-convex} show that, in each instance, convex approximation sets can be computed for multiobjective \emph{minimization and maximization} problems in (fully) polynomial time if and only if there is a (fully) polynomial-time approximation scheme for the weighted sum scalarization. \cite{Helfrich+etal:approximation-via-scalarization} derive sufficient and necessary conditions on general scalarizations such that, in each instance, optimal solution sets for this scalarization are approximation sets. 


\medskip

As already outlined in~\cite{Helfrich:mp-approximation}, computing optimal and approximate solution sets for the weighted sum scalarization\footnote{Convex approximation sets can be characterized via approximate solution sets for the weighted sum scalarization, see Proposition~\ref{prop:convex-approx=mult-para-approx}.} is closely related to the computation of optimal and approximate solutions sets, respectively, for so-called linear (multi-)parametric optimization problems. In such problems, the objective function depends affine-linearly on a single parameter (in the case of a parametric optimization problem) or on multiple parameters (in the case of multi-parametric optimization problems) and the goal is to provide, for any possible (combination of) parameter value(s), an optimal or approximate solution for the non-parametric optimization problem induced by fixing the parameter(s). A general solution approach for obtaining optimal solution sets for linear parametric optimization problems is presented by \citet{Eisner+Severance:method}. 
Exact solution methods for specific optimization problems exist for the linear parametric shortest path problem~\citep{Karp+Orlin:parametric-shortest-path}, the linear parametric assignment problem~\citep{Gassner+Klinz:parametric-assignment}, and the linear parametric knapsack problem~\citep{Eben-Chaime:par-knapsack}. Exact solution methods for general linear  multi-parametric optimization problems are studied in~\citep{Gass1955,Saaty1954,Gal1972,Borrelli2003}. For a recent survey on more general multi-parametric optimization problems and corresponding solution methods, we refer to~\cite{Oberdieck2016}.

Again, the minimum number of solutions in an optimal solution set for a (multi-) parametric optimization problem can be super-polynomially large in the instance size~\citep{Carstensen:PHD}. Hence, many linear (multi-) parametric optimization problems do not admit polynomial-time algorithms in general, even if $\textsf{P}=\textsf{NP}$ and only one parameter is considered. This fact strongly motivates the design of approximation algorithms.  A generic algorithm for linear parametric optimization problems, which can be interpreted as an approximate version of the method of Eisner and Severance, is presented in~\cite{Bazgan+etal:parametric}. The approximation of the linear parametric 0-1-knapsack problem is considered in~\cite{Giudici+etal:param-knapsack,Halman+etal:param-weigth-knapsack,Holzhauser+Krumke:param-knapsack}. In case of linear multi-parametric optimization problems, a general approximation method is presented in~\cite{Helfrich:mp-approximation}.

\subsection{Our Contribution}
We present an algorithm to compute convex approximation sets that is applicable to all multiobjective \emph{minimization and maximization problems} for which efficient exact or approximate solution methods for the weighted sum scalarization are available. Given~$\varepsilon > 0$, our algorithm iteratively calls the exact or approximate algorithm for the weighted sum scalarization and
outputs sets of solutions with cardinality polynomially bounded in the instance size and~$\frac{1}{\varepsilon}$. These solution sets constitute convex approximation sets with approximation quality $1 + \varepsilon$ times the approximation quality achieved by the algorithm for the weighted sum scalarization. Consequently, our algorithm yields a multiobjective (fully) polynomial-time approximation scheme for computing convex approximation sets (an \emph{M(F)PTcAS}) if a polynomial-time exact algorithm or an (F)PTAS for the problems induced by the weighted sum scalarization is available. 

From a high-level perspective, our algorithm follows the principle of operation of the dual variant of Benson's Outer Approximation Algorithm~\citep{Ehrgott2012DualBenson,Boekler:output-sensitive}. However, we introduce rounding schemes for weight vectors (of the weighted sum scalarization) that  guarantee that the algorithm for the weighted sum scalarization is only called for weight vectors contained in a polynomial-sized multiplicative grid as introduced in~\cite{Helfrich:mp-approximation}. 
Hence, in contrast to the convex approximation algorithms presented in~\cite{Diakonikolas:thesis,Helfrich:mp-approximation}, the algorithm for the weighted sum scalarization is called adaptively, i.e., based on information obtained during the approximation process, which yields significant improvements in terms of practical running time and reduced cardinality of the returned solution sets while still guaranteeing the same worst-case convex approximation quality. 
We substantiate these advantages by the first performance study of all existing convex approximation algorithms using instances of the triobjective knapsack problem and the triobjective symmetric metric traveling salesman problem.

In Section~\ref{sec:prelim}, we introduce basic notation and definitions concerning multiobjective optimization, the weighted sum scalarization, and approximation. In particular, we contrast the (classic) concept of approximation sets~\citep{Papadimitriou+Yannakakis:multicrit-approx} and the concept of convex approximation sets~\citep{Diakonikolas:thesis,Diakonikolas+Yannakakis:epsilon-convex}, which can be characterized by means of approximate solutions for the weighted sum scalarization.
In Section~\ref{sec:Foundations}, we lay the theoretical foundation for our algorithm: we recall the dual variant of Benson's Outer Approximation Algorithm~\citep{Ehrgott2012DualBenson,Boekler:output-sensitive} and the grid approach~\citep{Helfrich:mp-approximation} in Section~\ref{sec:dual-benson} and Section~\ref{sec:grid}, respectively. Further, in Section~\ref{sec:rounding-schemes}, we introduce the two rounding schemes that constitute the most essential building blocks of our algorithm. 
Subsequently, in Section~\ref{sec:Approx+Convex:oaa-algorithm}, we state our convex approximation algorithm. We further prove correctness and analyze its worst-case running time. The performance study of all known convex approximation algorithms is then presented in Section~\ref{sec:CompStudy}. 

Note that, in order to make this paper as self-contained as possible, we present the fundamental theoretic results in full, though parts are already published in~\cite{Diakonikolas:thesis,Helfrich:mp-approximation}.

\section{Preliminaries}\label{sec:prelim}
In this section, we introduce basic notation and definitions concerning multiobjective optimization, the weighted sum scalarization, approximation sets and algorithms, and convex approximation sets and algorithms. We further characterize convex approximation sets by means of approximate solutions for the weighted sum scalarization.

\medskip

In the following, we use the notation~$\R^d_{\geqq} \coloneqq \{y \in \R^d \st 0 \leqq y\}$, where $0 \in \R^d$ is the $d$-dimensional zero vector and $\leqq$ is the weak component-wise order defined by
\begin{align*}
    y \leqq y' \text{ if and only if } y_i \leq y'_i \text{ for } i =1,\ldots,d.
\end{align*}
We consider general multiobjective optimization problems, which are defined as follows:
\begin{definition}\label{def:MOP}
    For $d \geq 1$, a \emph{$d$-objective optimization problem~$\Pi$} is given by a set of instances. Each instance~$\I = (X,f)$ consists of a (finite or infinite) non-empty set~$X$ of \emph{feasible solutions} and a vector~$f = (f_1,\ldots, f_d)$ of $d$ \emph{objective functions}~$f_i : X \to \R$, $i = 1,\ldots,d$. If all objective functions are to be minimized, $\Pi$ is called a \emph{$d$-objective minimization problem} and if all objective functions are to be maximized, $\Pi$ is called a \emph{$d$-objective maximization problem}.
\end{definition}
Note that the set of feasible solutions might not be given explicitly. An image of a feasible solution is called a \emph{feasible image}, and $Y\coloneqq f(X) \coloneqq \{ f(x) : x \in X\} \subseteq \R^d$ is called the \emph{image set}. 
We consider optimality as defined as follows:
\begin{definition}
    In an instance~$\I = (X,f)$ of a minimization (maximization) problem, a solution~$x \in X$ \emph{dominates} another solution~$x' \in X$ if $f(x) \neq f(x')$ and $f(x) \leqq f(x')$ ($f(x) \geqq f(x')$). 
    A solution~$x \in X$ is called \emph{efficient} if there exists no other solution~$x' \in X$ that dominates~$x$. In this case, we call the corresponding image~$y = f(x) \in Y$ a \emph{nondominated image}. The set~$X_E \subseteq X$ of all efficient solutions is called the \emph{efficient set} and the set~$Y_N \coloneqq f(X_E)$ of nondominated images is called the \emph{nondominated set}.
\end{definition}
Given an instance~$\I = (X,f)$, the typical goal in multiobjective optimization is to return a set~$P^*\subseteq X$ that contains, for each nondominated image~$y \in Y_N$, a corresponding efficient solution~$x \in X_E$ with $f(x) = y$.

\medskip

In this work, we follow rather standard assumption of rational, nonegative-valued, polynomially computable objective functions made in the context of approximation for multiobjective optimization problems~\citep{Vassilvitskii+Yannakakis:trade-off-curves,Diakonikolas:thesis,Diakonikolas+Yannakakis:approx-pareto-sets}: 
\begin{assumption}\label{ass:mo}
	The number of objective functions~$d$ is assumed to be constant. Moreover, for any multiobjective optimization problem~$\Pi$, there exists a polynomial~$\textup{pol}$ such that, for any instance~$\I = (X,f)$ of $\Pi$, there exists a constant~$m \leq \textup{pol}(\textup{enc}(\I ))$ such that $\textup{enc}(f_i(x)) \leq m$ for each  solution~$x \in X$ and for each $i \in \{1,\ldots,d\}$, where $\textup{enc}(\I )$ and $\textup{enc}(f_i(x))$ denote the encoding lengths of the instance~$\I$ and the value~$f_i(x)$, respectively. The values~$f_1(x),\ldots,f_d(x)$ are further assumed to be nonnegative. All these assumptions, in particular, imply that $f_i(x) = 0$ or $2^{-m} \leq f_i(x) \leq 2^m$ for all~$x \in X$ and all $i \in \{ 1,\ldots,d \}$, and any two values~$f_i(x)$ and $f_i(x')$ differ by at least~$2^{-2m}$ if they are not equal. The bounds $2^{-m}$ and $2^m$ are often rather pessimistic. Hence, we additionally assume that we can compute positive rational (instance-dependent) bounds~$\LB \geq 2^{-m}$ and~$\UB \leq 2^m$ in polynomial time such that $f_i(x) \in \{0\} \cup [\LB, \UB]$ for all~$x \in X$ and all $i \in \{ 1,\ldots,d \}$. 
\end{assumption}
Note that these technical assumption are rather mild and satisfied for a large variety of multiobjective optimization problems such as combinatorial optimization problems and (mixed-) integer linear programs with a bounded feasible set, see, for example,~\cite{Bazgan+etal:parametric,Helfrich:mp-approximation} and Section~\ref{sec:CompStudy}.

\medskip

\emph{Scalarizations}, in particular the weighted sum scalarization that transforms a multiobjective optimization problem by means of a weighted sum of the objectives, are often a key building block for obtaining (efficient) solutions. 
\begin{definition}\label{def:weighted-sum}
    For an instance~$\I = (X,f)$ of a $d$-objective minimization (maximization) problem and a weight vector~$\lambda = (\lambda_1,\ldots, \lambda_p) \in \R^d_\geqq \setminus \{0\}$, the \emph{weighted sum scalarization with weight vector $\lambda$} is the single-objective instance
    \begin{align*}
        \min_{x \in X} \left(\max_{x \in X} \right) ~ \lambda_1 \cdot f_1(x) + \ldots  + \lambda_d \cdot f_d(x).
    \end{align*} 
    Then, a solution~$x \in X$ is called \emph{optimal} for $\lambda$ if $\sum_{i=1}^d \lambda_i \cdot f_i(x) \leq \sum_{i=1}^d \lambda_i \cdot f_i(x')$ ($\sum_{i=1}^d \lambda_i \cdot f_i(x) \geq \sum_{i=1}^d \lambda_i \cdot f_i(x')$) for all $x' \in X$. A set of solutions~$S^* \subseteq X$ is an \emph{optimal solution set} for the weighted sum scalarization if, for each $\lambda \in \R^d_\geqq \setminus \{0\}$, there exists a solution~$x^\lambda \in S^*$ that is optimal for $\lambda$. Further, solutions that are optimal for some~$\lambda \in \R^d_\geqq \setminus \{0\}$ are called \emph{supported}.
\end{definition}
It is well-known that a solution is efficient if it is optimal for a positive weight vector. 
If the weight vector contains zero-valued components, there exists at least one optimal solution that is efficient (under our assumptions). However, not every efficient solution must be supported~\citep{Ehrgott:book}.

\medskip

Algorithms for the construction of optimal solution sets for the weighted sum scalarization are widely studied in the literature~\citep{Alves2016esn,Aneja1979bicritransport,Benson1998OuterApproximation,Boekler:output-sensitive,Cohon2004multiobjectiveprogramming,Ehrgott2012DualBenson,Halffmann2020,Przybylski2010ESN,Oezpeynirci2010ESN}.
However, many multiobjective optimization problems are \emph{intractable} in the sense that the cardinality of the nondominated image set can be exponentially large in the instance size. This includes, for example, the multiobjective shortest path problem, the multiobjective assignment problem, the multiobjective knapsack problem, and the multiobjective traveling salesman problem, see~\cite{Figueira2017EasyMoco} for a more detailed study on tractability. This intractability often holds even true for the minimum cardinality of optimal solution sets for the weighted sum scalarization~\citep{Carstensen:PHD,Gassner+Klinz:parametric-assignment,Nikolova+etal:ESA06,Ruhe:parametric-network-flows}. So, for many important problems, the existence of polynomial-time exact algorithms as well as efficient algorithms for determining optimal solution sets for the weighted sum scalarization is ruled out, even if $\textsf{P} = \textsf{NP}$. This strongly motivates the concept of approximation:
\begin{definition}
    Let~$\alpha\geq 1$. In an instance~$\I = (X,f)$ of a $d$-objective minimization (maximization) problem, a feasible solution~$x \in X$ \emph{$\alpha$-approximates} another feasible solution~$x' \in X$ if $f(x) \leqq \alpha \cdot f(x')$ ($f(x) \geqq  \frac{1 }{\alpha} \cdot f(x')$). 
    A set~$P_{\alpha} \subseteq X$ of feasible solutions is called an \emph{$\alpha$-approximation set} if, for each feasible solution~$x'\in X$, there exists a solution~$x \in P_{\alpha}$ that $\alpha$-approximates $x'$. For $\varepsilon > 0$, a $(1 + \varepsilon)$-approximation set is also referred to as an \emph{$\varepsilon$-Pareto set}.
\end{definition}
\begin{definition}
	For~$\alpha \geq 1$, an \emph{$\alpha$-approximation algorithm}~$\A_\alpha$ for a $d$-objective minimization or maximization problem is an algorithm that, given an instance~$\I = (X,f)$, returns an $\alpha$-approximation set for~$\I$ in polynomial time. A \emph{multiobjective polynomial-time approximation scheme (MPTAS)} is a family~$\{\A_{1+ \varepsilon} \st \varepsilon>0\}$ of algorithms such that, for every~$\varepsilon>0$, the algorithm~$\A_{1+\varepsilon}$ is a $(1+\varepsilon)$-approximation algorithm. If the running time of each~$\A_{1+\varepsilon}$ is, additionally, polynomial in $\frac{1}{\varepsilon}$, the family~$\{\A_{1+ \varepsilon} \st \varepsilon>0\}$ is called a \emph{multiobjective fully polynomial-time approximation scheme (MFPTAS)}.
\end{definition}
\cite{Papadimitriou+Yannakakis:multicrit-approx} show that, for every~$\varepsilon > 0$ and in instances meeting our assumptions, there always exist polynomial-sized $\varepsilon$-Pareto sets. Further, they show that the efficient computability of $\varepsilon$-Pareto sets is polynomially equivalent to the efficient solvability of the so-called \emph{gap problem}, which is an approximate variant of the canonical decision problem associated with the multiobjective optimization problem. However, there also exist problems such as the biobjective minimum $s$-$t$-cut problem for which no MFPTAS can exist unless $\textsf{P} = \textsf{NP}$~\citep{Papadimitriou+Yannakakis:multicrit-approx}. This changes when a relaxed concept of approximation as introduced in~\cite{Diakonikolas+Yannakakis:epsilon-convex,Diakonikolas:thesis} is considered:
\begin{definition}[\cite{Diakonikolas+Yannakakis:epsilon-convex}]\label{def:convex-approximation-multi-objective}
	Let~$\alpha\geq 1$. In an instance~$\I = (X,f)$ of a $d$-objective minimization (maximization) problem, a set $S_\alpha \subseteq X$ of feasible solutions is an \emph{$\alpha$-convex approximation set} if, for each feasible solution~$x' \in X$, there exist an integer~$L \geq 1$, solutions $x^1,\ldots, x^L \in S_\beta$ 
    and scalars~$\theta_1,\ldots,\theta_L \geq 0$ with $\sum_{\ell = 1}^L \theta_\ell = 1$ such that $\sum_{\ell = 1}^L \theta_\ell \cdot f(x^\ell) \leqq \alpha \cdot f(x')$ ( $\sum_{\ell = 1}^L \theta_\ell \cdot f(x^\ell) \geqq \frac{1}{\alpha} \cdot f(x') $).
	For~$\varepsilon>0$, a $(1+\varepsilon)$-convex approximation set is also referred to as an \emph{$\varepsilon$-convex Pareto set}.
\end{definition}
Figure~\ref{fig:approximation} contrasts these two concepts of approximation for an instance of a biobjective minimization problem.
\begin{figure}[tb]
	\centering
    \includegraphics[width=0.8\textwidth]{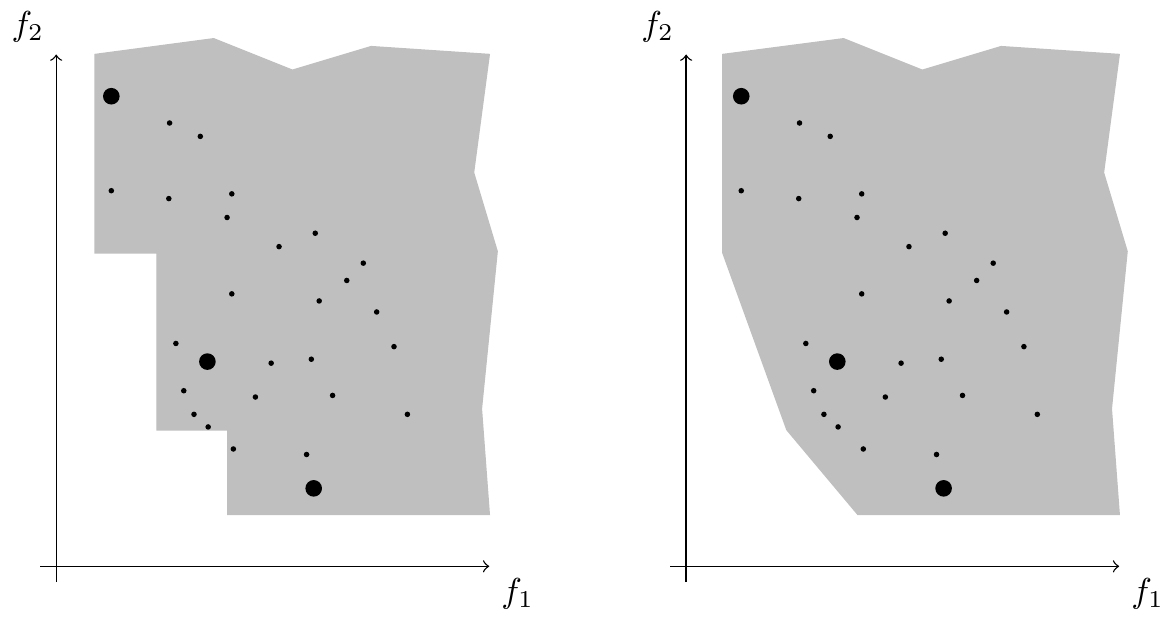}
	\caption{An $\alpha$-approximation set (left) and an $\alpha$-convex approximation set (right) for an instance of a biobjective minimization problem. The gray region indicates the location of all images approximated by (convex combinations of) the bold images.}
	\label{fig:approximation}
\end{figure}
We also define convex approximation algorithms:
\begin{definition}\label{def:convex-approximation-algorithm}
	For~$\alpha \geq 1$, an \emph{$\alpha$-convex approximation algorithm}~$\A_\alpha$ for a $d$-objective minimization or maximization problem is an algorithm that, given an instance~$\I$, returns an $\alpha$-convex approximation set for~$\I$ in polynomial time. A \emph{multiobjective polynomial-time convex approximation scheme (MPTcAS)} is a family~$\{\A_{1+ \varepsilon} \st \varepsilon>0\}$ of algorithms such that, for every~$\varepsilon>0$, the algorithm~$\A_{1+\varepsilon}$ is a $(1+\varepsilon)$-convex approximation algorithm. If the running time of each~$\A_{1+\varepsilon}$ is, additionally, polynomial in $\frac{1}{\varepsilon}$, the family~$\{\A_{1+ \varepsilon} \st \varepsilon>0\}$ is called a \emph{multiobjective fully polynomial-time convex approximation scheme (MFPTcAS)}.
\end{definition}
Since every $\alpha$-approximation set is also an $\alpha$-convex approximation set, any M(F)PTAS is also an M(F)PcAS. Clearly, the converse does not hold. The following proposition shows that convex approximation sets can be characterized by means of the weighted sum scalarization. The proposition extends~\cite{Diakonikolas:thesis}, where the result is shown under the additional assumption that all feasible images are approximately balanced, i.e., all components of an image are within a ratio of~$2$ of each other.
\begin{proposition}\label{prop:convex-approx=mult-para-approx}
	Let~$\alpha \geq 1$. In an instance~$\I = (X,f)$ of a $d$-objective minimization/maximization problem, a set of feasible solutions~$S_\alpha \subseteq X$ is an $\alpha$-convex approximation set if and only if, for each weight vector~$\lambda \in \R^d_\geqq \setminus \{0\}$, there exists a solution~$x^\lambda \in S_\alpha$ such that
	\begin{align}
	\lambda^\top f(x^\lambda) &\leq \alpha \cdot \lambda^\top f(x) \text{ for all } x \in X \text{ (minimization)}, \label{eq:char-convex-approx} \\
	\lambda^\top f(x^\lambda) &\geq \frac{1}{\alpha} \cdot \lambda^\top f(x) \text{ for all } x \in X \text{ (maximization)}. \notag
	\end{align}
\end{proposition}
\begin{proof} We prove the claim for the minimization case. The maximization case can be handled similarly.
	Let $S_\alpha$ be an $\alpha$-convex approximation set, let $\lambda \in \R^d_\geqq \setminus \{0\}$, and let~$x^\lambda \in S_\alpha$ such that $ \lambda^\top f(x^\lambda) \leq \lambda^\top f(x)$ for all~$x \in S_\alpha$. Hereby, note that Assumption~\ref{ass:mo} implies that the image set of the given instance is finite and, thus, such a solution~$x^\lambda$ exists. Since, for each solution~$x \in X$, there exist solutions~$x^1,\ldots, x^L \in S_\alpha$ and scalars~$\theta_1,\ldots, \theta_L \geq 0$, $\sum_{\ell = 1}^L \theta_\ell = 1$, such that $\sum_{\ell = 1}^L \theta_\ell \cdot f(x^\ell) \leqq \alpha \cdot f(x)$, it follows that
	\begin{align*}
	\lambda^\top f(x^\lambda) = \sum_{\ell = 1}^L \theta_\ell \cdot \lambda^\top f(x^\lambda) \leq \sum_{\ell = 1}^L \theta_\ell \cdot \lambda^\top f(x^\ell) \leq \alpha \cdot \lambda^\top f(x).
	\end{align*}
	Now, let~$S_\alpha \subseteq X$ be a set of feasible solutions such that, for each~$\lambda \in \R^d \setminus \{0\}$, there exists a solution~$x^\lambda \in S_\alpha$ with
	$\lambda^\top f(x^\lambda) \leq \alpha \cdot \lambda^\top f(x)$ for all $x \in X$. Then, each solution~$x \in X$ such that
	\begin{align*}
	f(x) \in Q &\coloneqq \conv\left( \left\{ \frac{1}{\alpha} f(\hat{x}) \st \hat{x} \in S_\alpha \right\} \right) + \R^d_\geqq \\
	&= \frac{1}{\alpha} \cdot \conv\left( \left\{ f(\hat{x}) \st \hat{x} \in S_\alpha \right\} \right) + \R^d_\geqq \subseteq \R^d_\geqq
	\end{align*}
	is $\alpha$-approximated by a convex combination of finitely many images of solutions in~$S_\alpha$.
	
	Assume for the sake of a contradiction that there exists a solution~$x^* \in X$ such that $f(x^*) \notin Q$. Then, by the strict hyperplane separation theorem~\citep{Boyd+Vandenberghe:convex-opt}, there exists a vector~$\lambda \in \R^d$ such that $\lambda^\top f(x^*) < \lambda^\top q$ for all points~$q \in Q$. Since $q + t \cdot e^i \in Q$ for all $q \in Q$, $i = 1,\ldots,d$, and all $t \geq 0$, where~$e^i$ denotes the $i$th unit vector in $\R^d$, it holds that $\lambda \in \R^d_\geqq \setminus \{0\}$. Hence, there exists a solution~$x^\lambda \in S_\alpha$ such that $\lambda^\top f(x^\lambda) \leq \alpha \cdot \lambda^\top f(x^*)$. Since~$\frac{1}{\alpha} \cdot f(x^\lambda) \in  Q$, it follows that
	\begin{align*}
	\lambda^\top f(x^*) < \lambda^\top \left(\frac{1}{\alpha} \cdot f(x^\lambda)\right) \leq \frac{\alpha \cdot \lambda^\top f(x^*)}{\alpha} = \lambda^\top f(x^*),
	\end{align*}
	which is a contradiction.
\end{proof}
Consequently, as indicated in Inequality~\eqref{eq:char-convex-approx}, convex approximation sets relate to approximate solutions of the weighted sum scalarization. Hence, we define:
\begin{definition}
    Let $\alpha \geq 1$ and $\lambda \in \R^d_\geqq \setminus \{0\}$. In an instance~$\I = (X,f)$ of a $d$-objective minimization (maximization) problem, a solution~$x \in X$ is called an \emph{$\alpha$-approximation for $\lambda$} if $\sum_{i=1}^d \lambda_i \cdot f_i(x) \leq \alpha \cdot \sum_{i=1}^d \lambda_i \cdot f_i(x')$ ($\sum_{i=1}^d \lambda_i \cdot f_i(x) \geq \frac{1}{\alpha} \cdot \sum_{i=1}^d \lambda_i \cdot f_i(x')$) for all $x' \in X$. 
\end{definition}
Hence, a set~$S_\alpha \subseteq X$ of solutions is an $\alpha$-convex approximation set if and only if it contains, for each $\lambda \in \R^d_\geqq \setminus \{0\}$, a solution~$x^\lambda \in S_\alpha$ that is an  $\alpha$-approximation for~$\lambda$. 
Clearly, for each solution~$x \in X$, there is a (possibly empty) subset~$\Lambda' \subseteq \R^d_\geqq \setminus \{0\}$ of weight vectors such that~$x$ is an $\alpha$-approximation for all weight vectors~$\lambda' \in \Lambda'$. Hence, this notion of is relaxed as follows: A solution~$x$ is an $\alpha$-approximation for $\Lambda' \subseteq \R^d_\geqq \setminus \{0\}$ if it is an $\alpha$-approximation for every~$\lambda' \in \Lambda'$. Finally, we define approximation algorithms for the weighted sum scalarization:
\begin{definition}
Let $\alpha \geq 1$. An \emph{$\alpha$-approximation algorithm~$\ALG_\alpha$ for the weighted sum scalarization} is an algorithm that, given an instance~$\I$ and a weight vector~$\lambda \in \R^d \setminus \{0\}$, returns an $\alpha$-approximation for $\lambda$ in polynomial time. A \emph{polynomial-time approximation scheme (PTAS) for the weighted sum scalarization} is a family~$\{\ALG_{1 + \varepsilon} : \varepsilon>0\}$ of algorithms such that, for every~$\varepsilon>0$, the algorithm~$\ALG_{1 + \varepsilon}$ is a $(1+\varepsilon)$-approximation algorithm for the weighted sum scalarization. If the running time of each~$\ALG_{1 + \varepsilon}$ is, additionally, polynomial in~$\frac{1}{\varepsilon}$, the family~$\{\ALG_{1 + \varepsilon} : \varepsilon>0\}$ is called a \emph{fully polynomial-time approximation scheme (FPTAS) for the weighted sum scalarization}.
\end{definition}
With Inequality~\eqref{eq:char-convex-approx}, it is easy to see that the existence of an M(F)PTcAS implies the existence of an (F)PTAS for the weighted sum scalarization. 
Moreover,~\cite{Diakonikolas:thesis} presents, under the additional assumption that all objective function values are strictly positive, an M(F)PTcAS that is based on scaling images to be \emph{approximately balanced} and based on an (F)PTAS for the weighted sum scalarization.
The algorithm presented in~\cite{Helfrich:mp-approximation} lifts approximation algorithms for the weighted sum scalarization to construct convex approximation sets with approximation qualities that are arbitrarily close to the approximation quality of the weighted sum approximation algorithm. Since this algorithm avoids scaling the images to be approximately balanced and is, thus, applicable to instances with nonnegative image sets as well, we can conclude that 
there exists an M(F)PTcAS for a multiobjective optimization problem if and only if there is an (F)PTAS for the weighted sum scalarization. So, recalling  the example from above, there exists an MFPTcAS for the multiobjective min $s$-$t$-cut problem since its single-objective version can be solved exactly in polynomial time~\citep{Goldberg1988:max-flow}.

\medskip

Both algorithms in \cite{Diakonikolas:thesis,Helfrich:mp-approximation} call the algorithm for the weighted sum scalarization in a non-adaptive way. That is, no information induced by solutions obtained during the process is used to cleverly choose which scalarized optimization problem should be solved next.
In this article, we present an algorithm that adaptively chooses the scalarized optimization problems to be solved in order to the decrease practical running time as well as the cardinality of the returned convex approximation set. Our algorithm combines the idea of the dual variant of Benson's Outer Approximation Algorithm~\citep{Boekler:output-sensitive,Ehrgott2012DualBenson} with the grid approach presented in~\cite{Helfrich:mp-approximation}. As a consequence of this, our algorithm can be interpreted as an approximate version of the dual variant of Benson's Outer Approximation Algorithm~\citep{Boekler:output-sensitive,Ehrgott2012DualBenson} and/or as a multiobjective generalization of the approximate variant of the dichotomic search algorithm for biobjective optimization problems introduced in~\cite{Daskalakis+etal:Chord-Algorithm,Bazgan+etal:parametric}. As in \cite{Helfrich:mp-approximation}, our algorithm lifts an algorithm~$\ALG$ for the weighted sum scalarization to obtain a convex approximation set with approximation quality arbitrarily close to the approximation quality of~$\ALG$ while relying on polynomially many calls to~$\ALG$ only. In particular, if $\ALG$ is a polynomial-time exact algorithm or an (F)PTAS for the weighted sum scalarization, our algorithm constitutes an M(F)PTcAS. To this end, our last assumption is naturally as follows:
\begin{assumption}\label{ass:approx-alg-weighted-sum}
	For some~$\alpha \geq 1$, there exists an $\alpha$-approximation algorithm~$\ALG_{\alpha}$ for the weighted sum scalarization, i.e., $\ALG_{\alpha}$ returns, for each instance~$\I = (X,f)$ and for each weight vector~$\lambda \in \R^d_\geqq \setminus \{0\}$, a solution~$x'$ such that $\lambda^\top f(x') \leq \alpha \cdot \lambda^\top f(x)$ for all~$x \in X$ in the case of a minimization problem or $\lambda^\top f(x') \geq \frac{1}{\alpha} \cdot \lambda^\top f(x)$ for all~$x \in X$ in the case of a maximization problem.\footnote{The approximation guarantee~$\alpha$ is assumed to be independent of~$\lambda$. However, it is allowed that~$\alpha$ depends on the instance (such that the encoding length of~$\alpha$ is polynomially bounded in the encoding length of the instance).} The running time of $\ALG_{\alpha}$ is denoted by~$T_{\ALG_{\alpha}}$.
\end{assumption}

\section{Foundations}\label{sec:Foundations}
In this section, we present the theoretical foundations for our algorithm. In Section~\ref{sec:dual-benson}, we first recall the dual variant of Benson's Outer Approximation Algorithm, which uses an exact algorithm for the weighted sum scalarization to compute exact solution sets for the weighted sum scalarization. Our algorithm modifies this algorithm to compute convex approximation sets. This is done by introducing rounding schemes for weight vectors that  guarantee that the algorithm for the weighted sum scalarization is only called for weight vectors that are contained in a polynomial-sized multiplicative grid on the set of all eligible weights.  So, in a nutshell, our algorithm combines the idea of the dual variant of Benson's Outer Approximation Algorithm with the grid approach presented in~\cite{Helfrich:mp-approximation}.
The polynomial-sized multiplicative grid is introduced in Section~\ref{sec:grid}. The rounding schemes are then introduced in Section~\ref{sec:rounding-schemes}.

\medskip

To this end, observe that a solution~$x$ is optimal for $\lambda$ if and only if $x$ is optimal for~${t \cdot \lambda}$ for every positive scalar~$t > 0$. This holds also in the approximate sense: For~$\beta \geq 1$, a solution~$x$ is a $\beta$-approximation for $\lambda$ if and only if $x$ is a $\beta$-approximation for $t \cdot \lambda$ for every positive scalar~$t > 0$. This implies that the set of all eligible weight vectors on which the multiplicative grid is imposed can be restricted to the \emph{weight set}~\citep{Przybylski2010ESN}$$\Lambda \coloneqq \{ \lambda \in \R^d_\geqq \st \norm{\lambda}_1 \coloneqq \sum_{i=1}^d \lambda_i = 1 \}.$$
Consequently, in order to obtain a $\beta$-convex approximation set for $\beta \geq 1$, it suffices to construct a set of solutions that contains a $\beta$-approximation for each $\lambda \in \Lambda$. The weight set~$\Lambda$ is a $(d-1)$-dimensional polytope in $\R^d_\geqq$. So, in the case~$d = 3$, we can illustrate weight vectors~$\lambda \in \Lambda$ and subsets~$\Lambda' \subseteq \Lambda$ of weight vectors in $\R^2$. This convenience is  illustrated in Figure~\ref{fig:weight-set}.
\begin{figure}[tb]
    \centering
    \includegraphics[width = .9\textwidth]{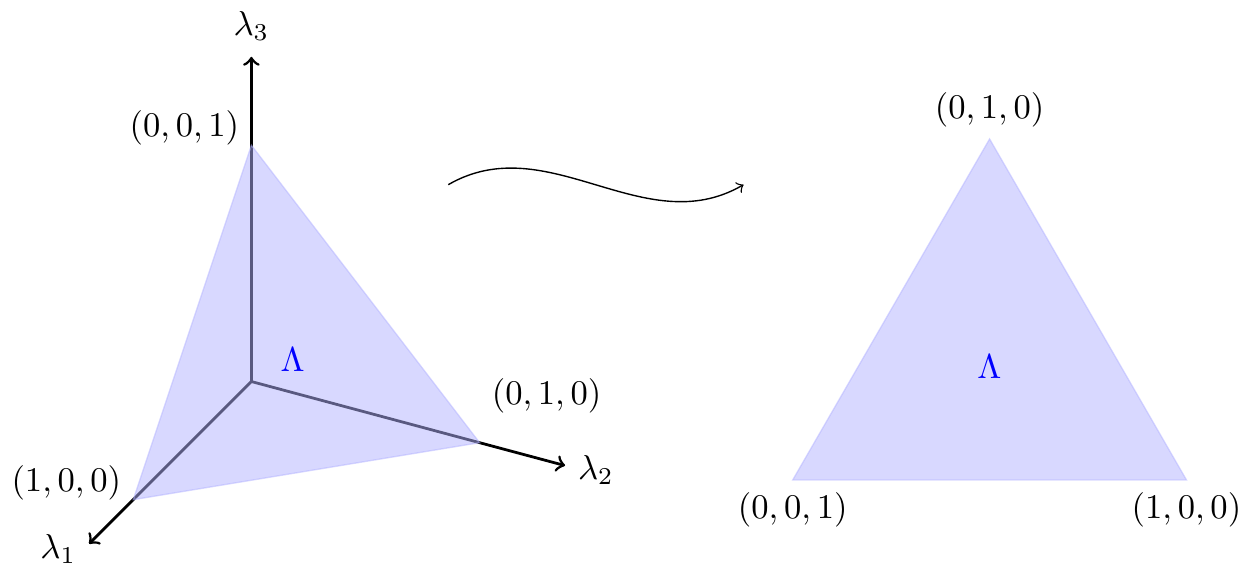}
    \caption{Illustration of the weight set~$\Lambda \subseteq \R^3$ and its projection onto~$\R^2$.}
    \label{fig:weight-set}
\end{figure}

\subsection{Dual variant of Benson's Outer Approximation Algorithm}\label{sec:dual-benson}
Next, we review the basic version of the dual variant of Benson's Outer Approximation Algorithm presented in~\cite{Ehrgott2012DualBenson} and adapted in \cite{Boekler:output-sensitive}. Note that variants exists that reduce the running time and/or guarantee a minimum cardinality of the returned solution set. However, the basic variant suffices for understanding our algorithm. 

\medskip

As described for example in~\cite{Ehrgott2012DualBenson}, we equip each weight vector~$\lambda \in \Lambda$ with an additional component~$z \in \R$ that will represent possible objective values of the weighted sum scalarization with weight vector~$\lambda$. That is, we lift the weight set and consider the Cartesian product $\Lambda \times \R$ of~$\Lambda$ and~$\R$ in which each feasible solution~$x \in X$ induces the half-space
\begin{align*}
H(x) \coloneqq \{ (\lambda_1,\ldots, \lambda_d, z) \in \Lambda \times \R : (\lambda_1,\ldots, \lambda_d)^\top f(x) \geq z\}
\end{align*}
in the case of minimization, and 
\begin{align*}
H(x) \coloneqq \{ (\lambda_1,\ldots, \lambda_d, z) \in \Lambda \times \R : (\lambda_1,\ldots, \lambda_d)^\top f(x) \leq z\}
\end{align*} 
in the case of maximization. Then, each finite set~$S$ of solutions induces the $d$-dimensional polyhedron 
\begin{align*}
	D(S) \coloneqq \bigcap_{x \in S} H(x),
\end{align*}
which is illustrated in Figure~\ref{fig:Approx+Convex:objective-polyhedron} 
\begin{figure}[tb]
	\centering
    \includegraphics[width = \textwidth]{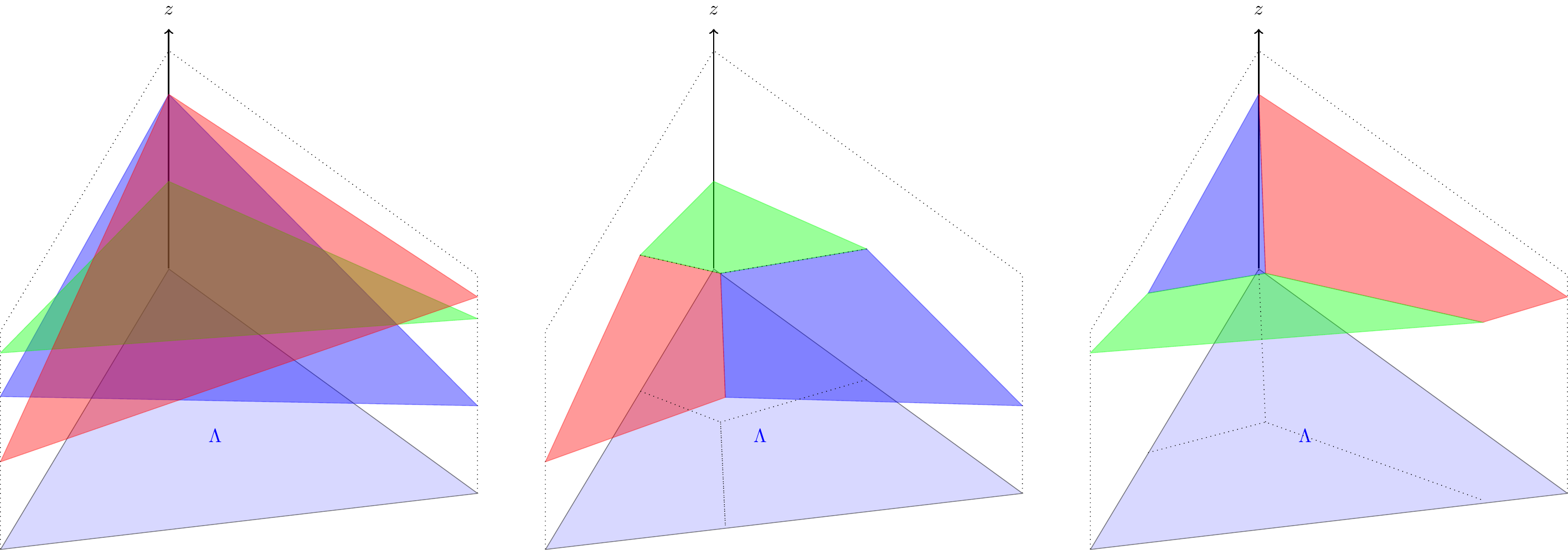}
	\caption{Half-spaces in $\Lambda \times \R$ induced by (left) three solutions~$x^1,x^2,x^3$ as well as the objective polyhedron of $S \coloneqq \{x^1,x^2,x^3\}$ in an instance of a 3-objective (middle) minimization  problem and (right) a maximization problem. }
	\label{fig:Approx+Convex:objective-polyhedron}
\end{figure}
and for which the following holds:
\begin{proposition}[\cite{Ehrgott2012DualBenson}]\label{prop:termination-DualBenson}
	Let~$S \subseteq X$ be a finite set of solutions. If every extreme point~$(\lambda_1, \ldots, \lambda_d,z)$ of $D(S)$ satisfies $ z = \min_{x \in X} \lambda^\top f(x)$, it follows that~$S$ is an optimal solution set for the weighted sum scalarization.
\end{proposition}
Hence, the basic version of the dual variant of Benson's Outer Approximation Algorithm works as follows: it starts by solving the weighted sum scalarization for the weight vector~$(\frac{1}{d},\ldots,\frac{1}{d})$. Based on the obtained solution~$x^*$, it initializes~$S \coloneqq \{x^*\}$, constructs the objective polyhedron~$D(S)$, and initializes a queue~$M$ containing all extreme points of $D(\{x^*\})$. Then, it iteratively removes an extreme point~$(\lambda_1,\ldots,\lambda_d,z)$ of~$D(S)$ from $M$, solves the weighted sum scalarization with weight vector~$\lambda = (\lambda_1,\ldots,\lambda_d)$, and obtains an optimal solution~$x'$. If $\lambda^\top f(x') = z$ and, thus, $\min_{x \in X} \lambda^\top f(x) = z$, it proceeds with the next iteration. Otherwise, it updates~$S = S \cup \{x'\}$ and re-initializes~$M$ by recomputing the extreme points of~$D(S)$ and then proceeds with the next iteration. The algorithm terminates if the queue~$M$ is empty, and returns an optimal solution set~$S$ for the weighted sum scalarization. 

\medskip

From a high-level perspective, our algorithm almost coincides with this basic variant of the dual variant of Benson's Outer Approximation Algorithm. However, approximation algorithms for the weighted sum scalarization can be utilized, a polynomially-sized multiplicative grid on the weight set is imposed, and two \enquote{rounding schemes} -- \texttt{BoundaryRounding} and \texttt{GridRounding} -- are introduced to guarantee termination, correctness, and polynomial running time. Both rounding schemes are applied to weight vectors \emph{before calling the algorithm for the weighted sum scalarization} such that the algorithm for the weighted sum scalarization is only called for weight vectors from the grid, which we concisely refer to as \emph{grid weight vectors} in the following. This allows us to polynomially bound the total number of calls to the algorithm for the weighted sum scalarization and, thus, to obtain a polynomial upper bound on the number of computed solutions. In particular, the latter yields polynomial running time of the vertex enumeration and, consequently, polynomial running time of the whole procedure. The characteristics of the underlying grid and the constructions of the rounding schemes also guarantee that the returned set~$S$ containing all computed feasible solutions is a convex approximation set with approximation quality arbitrarily close to the approximation quality of the approximation algorithm for weighted sum scalarization. In the next section, we explain the underlying polynomial-sized multiplicative grid. The subsequent section introduces the two rounding schemes.

\subsection{Multiplicative Grid}\label{sec:grid}
Next, we introduce the underlying polynomial-sized multiplicative grid. The construction essentially follows the approach in~\cite{Helfrich:mp-approximation}. Two results that connect the weight set with the concept of convex approximation are essential. In the following, we call these two results \emph{Continuity Property} and \emph{Projection Property} and first explain how these properties interact. 

\medskip

The Continuity Property ensures that imposing a multiplicative grid on the weight set is a reasonable approach:
\begin{property}[Continuity Property, \cite{Helfrich:mp-approximation}]
	\label{property:Approx+Convex:Continuity} Let~$\gamma \geq 1$, $\delta > 0$, and $\lambda' \in \Lambda$. Then, every $\gamma$-approximation for $\lambda'$ is a $((1 + \delta) \cdot \gamma)$-approximation for $$\Ball{\lambda'}{\gamma}{\delta} \coloneqq \left\{ \frac{\lambda}{\norm{\lambda}_1} \st \frac{1}{1 + \delta} \cdot \lambda'_i \leq \lambda_i \leq \lambda'_i, i = 1,\ldots, d \right\} \subseteq \Lambda.$$ 
\end{property}
Figure~\ref{fig:Approx+Convex:ContinuityPropoerty} illustrates different subsets~$\Ball{\lambda'}{\gamma}{\delta}$. So, roughly speaking, the guaranteed approximation quality of a $\gamma$-approximate solution~$x$ for a weight vector~$\lambda'$ decreases proportionally to the perturbation of the components of~$\lambda'$ in a multiplicative sense. The Continuity Property~\ref{property:Approx+Convex:Continuity} immediately suggest the following adaption of the grid approach of~\cite{Papadimitriou+Yannakakis:multicrit-approx}: Given~$\ALG_{\alpha}$ and $\varepsilon>0$, define
\begin{align}\label{eq:Approx+Convex:grid-infinite}
	\preGrid \coloneqq \left\{ \frac{\lambda'}{\norm{\lambda'}_1} \st \lambda'_i \in \{0\} \cup \{ (1 + \varepsilon)^{a_i} \st a_i \in \Z \}, i =1,\ldots, d \right\} \subseteq \Lambda,
\end{align}
call~$\ALG_{\alpha}$ for each grid weight vector~$\lambda' \in \preGrid$, and collect all obtained solutions. It is easy to see that there exists, for each weight vector~$\lambda \in \Lambda$, a grid weight vector~$\lambda' \in \preGrid$ such that $\lambda \in \Ball{\lambda'}{\alpha}{\varepsilon}$. Hence, the Continuity Property~\ref{property:Approx+Convex:Continuity} yields that the obtained set is a $((1 + \varepsilon) \cdot \alpha)$-convex approximation set. However, the cardinality of $\preGrid$ is not finite due to the multiplicative character of the grid, so the running time of this approach is not even close to being polynomial in the instance size and $\frac{1}{\varepsilon}$ --  the algorithm does in fact not even need to terminate in general. 
\begin{figure}
	\centering
    \includegraphics[width = .8\textwidth]{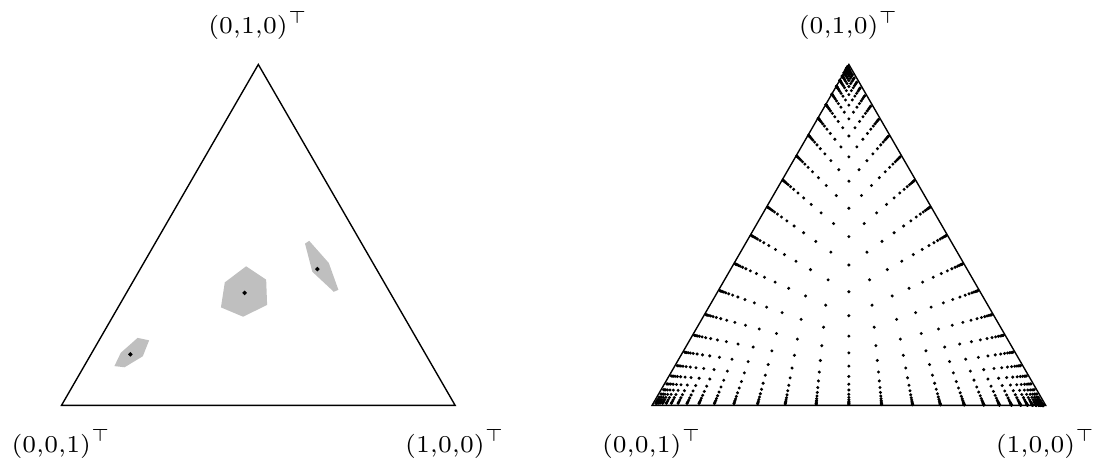}
	\caption{(left) Illustration of sets~$\Ball{\lambda}{\gamma}{\delta}$ (indicated in light gray) for $\gamma \geq 1$, $\delta > 0$ and different~$\lambda \in \Lambda$ (marked as black dots). Every~$\gamma$-approximation for $\lambda$ is a $((1 + \delta) \cdot \gamma)$-approximation for $\Ball{\lambda}{\gamma}{\delta}$. (right) The grid~$\preGrid$ with infinite cardinality as defined in \eqref{eq:Approx+Convex:grid-infinite}. }
 \label{fig:Approx+Convex:ContinuityPropoerty}
 \label{fig:Approx+Convex:InfiniteGrid}
\end{figure}

\medskip

Thus, in order to make this approach feasible, the grid must be (1) restricted to some compact subset~$\Compact \subseteq \int(\Lambda)$ to guarantee finite cardinality and (2) modified such that, for each weight vector~$\lambda \in \Lambda \setminus \Compact$, there still exists a grid weight vector~$\lambda'$ such that every~$\alpha$-approximation for $\lambda'$ is a $((1 + \varepsilon) \cdot \alpha)$-approximation for $\lambda$. 

In~\cite{Helfrich:mp-approximation}, Issue (1) is tackled by transferring the bounds~$\LB$ and~$\UB$ on positive objective function values of feasible solutions to the weight set in the following sense: 
\begin{property}[Projection Property, \cite{Helfrich:mp-approximation}]
	\label{property:Approx+Convex:projection}
	Let~$\beta \geq 1$ and~$0 < \varepsilon' <1$. Further, let~$\lambda \in \Lambda$ be a weight vector such that, for at least one index set~$\emptyset \neq I \subsetneq \{1,\ldots,d\}$, it holds that $\sum_{i \in I} \lambda_i < \frac{\varepsilon' \cdot \LB}{\beta \cdot \UB} \cdot \min_{j \notin I} \lambda_j$. Then there exists a weight vector~$\lambda' \in \Lambda$ satisfying~$\sum_{i \in I} \lambda_i \geq \frac{\varepsilon' \cdot \LB}{\beta \cdot \UB} \cdot \min_{j \notin I} \lambda_j$ for \emph{all} index sets~$\emptyset \neq I \subsetneq \{1,\ldots,d\}$ such that every~$\beta$-approximation for~$\lambda'$ is a $(\beta + \varepsilon')$-approximation for $\lambda$.
\end{property}
This means that weight vectors with components that sum up to a small threshold (i.e., the weight vectors that are close to the boundary of~$\Lambda$) can be neglected while still guaranteeing a slightly worsened approximation quality. Therefore, it is sufficient to construct a grid over the set
\begin{align*}\Compact \coloneqq \left\{ \lambda \in \Lambda \st  \sum_{i \in I} \lambda_i \geq \frac{\varepsilon' \cdot \LB}{\beta \cdot \UB} \cdot \min_{j \notin I} \lambda_j \text{ for all } \emptyset \neq I \subsetneq \{1,\ldots,d\} \right\}.
\end{align*} 
Issue (2) is then tackled by refining the grid in order to compensate the loss in the guaranteed approximation quality due to the Projection Property~\ref{property:Approx+Convex:projection}. The following lemma implies that the set~$\Compact$ is indeed a compact subset of~$\int(\Lambda)$ and provides bounds on the components of its weight vectors. Figure~\ref{fig:Approx+Convex:Projectionregion} illustrates~$\Compact$.
\begin{lemma}\label{lem:Approx+Convex:bounds-on-weight-set}
	Let~$\beta \geq 1$, $0 < \varepsilon' < 1$, and~$\lambda \in \Lambda$. Then, $\sum_{i \in I} \lambda_i \geq \frac{\varepsilon' \cdot \LB}{\beta \cdot \UB} \cdot \min_{j \notin I} \lambda_j$ for all index sets~$\emptyset \neq I \subsetneq \{1,\ldots,d\}$ implies that $\lambda_i \geq \frac{1}{d!} \cdot \left(\frac{\varepsilon' \cdot \LB}{\beta \cdot \UB}\right)^{d-1}$ for $i = 1,\ldots, d$. In particular, $\Compact \subseteq \int(\Lambda)$ holds true.
\end{lemma}
\begin{proof}
	Without loss of generality, let~$\lambda \in \Lambda$ such that $\lambda_1 \leq \lambda_2 \leq \ldots \leq \lambda_d$. Otherwise, the objective functions may be reordered accordingly. Then, by assumption, 
	\begin{align*}
	\sum_{i = 1}^k \lambda_i \geq \frac{\varepsilon' \cdot \LB}{\beta \cdot \UB} \cdot  \min_{j \in \{k+1,\ldots,d\}} \lambda_j = \frac{\varepsilon' \cdot \LB}{\beta \cdot \UB} \cdot  \lambda_{k+1}
	\end{align*}
	holds true for $k = 1,\ldots, d-1$. Thus,
	with~$d \cdot \lambda_d \geq \sum_{i=1}^d \lambda_i = 1$, this implies that~$\lambda_d \geq \frac{1}{d}$ and, thus, for $k = 1, \ldots, d-1$,
 \begin{align*}
    \lambda_k \geq \frac{1}{k} \cdot \frac{\varepsilon' \cdot \LB}{\beta \cdot \UB} \cdot \lambda_{k+1} \geq \frac{1}{d} \cdot \frac{1}{d-1} \cdot \ldots \cdot \frac{1}{k} \cdot \left( \frac{\varepsilon' \cdot \LB}{\beta \cdot \UB} \right)^{d-k}.
    \end{align*}
	Since~$\frac{\varepsilon' \cdot \LB}{\beta \cdot \UB} \leq 1$, the claim follows.	
	Note that $\lambda_i \geq \frac{1}{d!} \cdot \left(\frac{\varepsilon' \cdot \LB}{\beta \cdot \UB}\right)^{d-1}$ for $i = 1,\ldots, d$ and $\sum_{i=1}^d \lambda_i = 1$ imply that $\lambda_i  = 1 - \sum_{j \neq i} \lambda_j \leq 1 - (d-1) \cdot \frac{1}{d!} \cdot \left(\frac{\varepsilon' \cdot \LB}{\beta \cdot \UB}\right)^{d-1} <1$. Consequently, $\Compact \subseteq \int(\Lambda)$.
 
\end{proof}
It remains to choose $\varepsilon'$, $\beta \geq 1$, and the grid size appropriately: Given~$\ALG_{\alpha}$ for some~$\alpha \geq 1$ and $0 < \varepsilon < 1$, set~$\varepsilon' \coloneqq \sqrt{1 + \varepsilon} - 1 < 1$ and~$\beta \coloneqq (1 + \varepsilon') \cdot \alpha$. Further, define
\begin{align}\label{eq:Approx+Convex:lb}
	\lb \coloneqq \frac{1}{d!} \left( \frac{\varepsilon' \cdot \LB}{\beta \cdot \UB} \right)^{d-1}
\end{align}
to be the bound on the components of the weight vectors~$\lambda \in \Compact$ provided by Lemma~\ref{lem:Approx+Convex:bounds-on-weight-set} and define the grid
\begin{align}\label{eq:Approx+Convex:grid-finite}
\Grid \coloneqq \begin{pmatrix*}[l]  & & \lambda'_i = (1 + \varepsilon')^{a_i}, \\
\frac{\lambda'}{\norm{\lambda'}_1}  & \st & \log_{1 + \varepsilon'}(\lb) \leq a_i \leq \log_{1+\varepsilon'}(1 - (d-1) \cdot \lb ) + 1 \\
& &a_i \in \Z, i =1,\ldots, d \end{pmatrix*} \subseteq \Lambda.
\end{align}
Note that grid is constructed based on $\varepsilon'$ instead of $\varepsilon$ as done in \eqref{eq:Approx+Convex:grid-infinite}. Figure~\ref{fig:Approx+Convex:Grid+Projectionregion} illustrates the grid~$\Grid$ over~$\Compact$. 
\begin{figure}
	\centering
    \includegraphics[width = .8\textwidth]{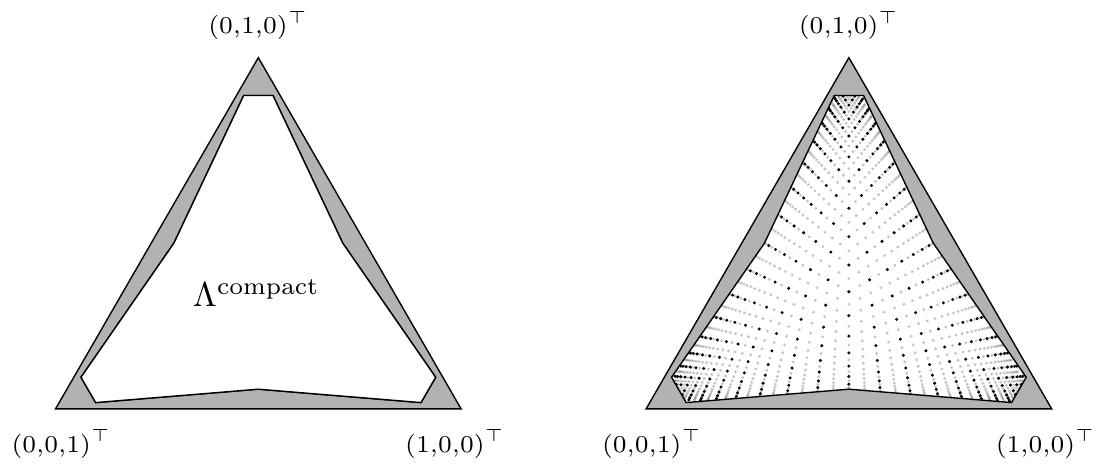}
	\caption{(left) Illustration of $\Compact$, see \cite{Helfrich:mp-approximation} for further details. (right) Illustration of the grid~$\Grid$ as defined in~\eqref{eq:Approx+Convex:grid-finite}. Grid weight vectors~$\lambda'$ of the (multiplicative) discretization based on $1 + \varepsilon$ are indicated with black dots. Light gray dots represent the remaining grid weight vectors of the (multiplicative) discretization based on $(1 + \varepsilon') = \sqrt{1 + \varepsilon}$.}
 \label{fig:Approx+Convex:Projectionregion}\label{fig:Approx+Convex:Grid+Projectionregion}
\end{figure}
The Continuity Property~\ref{property:Approx+Convex:Continuity} and the Projection Property~\ref{property:Approx+Convex:projection} can then be combined as follows:
\begin{enumerate}
	\item\label{item:Convex-Approx:weight-compact} For a weight vector~$\lambda \in \Compact$,  set~$a_i \coloneqq \lceil \log_{1+\varepsilon'}(\lambda_i) \rceil$ for $i = 1,\ldots,d$ and $\lambda' \coloneqq  ((1 + \varepsilon')^{a_1},\ldots, (1 + \varepsilon')^{a_d} )$. Then, since $\lb \leq \lambda_i \leq 1 - (d-1) \cdot \lb$, it follows that $\lambda' \in \Grid$. Moreover, $a_i - 1 \leq \log_{1+\varepsilon'}(\lambda_i) \leq a_i$ implies that
	\begin{align*}
		\frac{1}{1 + \varepsilon'} \cdot \lambda'_i = (1 + \varepsilon')^{a_i - 1} \leq \lambda_i \leq  (1 + \varepsilon')^{a_i} = \lambda'_i
	\end{align*}
	for $i = 1,\ldots, d$. Hence, $\lambda \in \Ball{\lambda'}{\varepsilon'}{\alpha}$ and, thus, every~$\alpha$-approximate solution for~$\lambda'$ is a $((1 + \varepsilon') \cdot  \alpha)$-approximation for $\lambda$. Since $1 + \varepsilon' \leq 1 + \varepsilon$, every~$\alpha$-approximate solution for~$\lambda'$ is  a $((1 + \varepsilon) \cdot  \alpha)$-approximation for $\lambda$.

	\item\label{item:Convex-Approx:weight-notin-compact} For a weight vector~$\lambda \in \Lambda \setminus \Compact$, let~$\bar{\lambda} \in \Compact$ be the weight vector obtained by the Projection Property~\ref{property:Approx+Convex:projection}. By~\ref{item:Convex-Approx:weight-compact}, there exists a grid weight vector~$\lambda' \in \Grid$ such that every $\alpha$-approximation for $\lambda'$ is a $((1 + \varepsilon') \cdot \alpha)$-approximation for~$\bar{\lambda}$. Consequently, by the choice of $\beta = (1 + \varepsilon') \cdot \alpha$ and the Projection Property~\ref{property:Approx+Convex:projection}, every $\alpha$-approximation for $\lambda'$ is a $((1 + \varepsilon') \cdot \alpha + \varepsilon')$-approximation for $\lambda$. Since
	\begin{align*}
		(1 + \varepsilon') \cdot \alpha + \varepsilon' \leq   (1 + \varepsilon') \cdot (1 + \varepsilon') \cdot \alpha = (1 + \varepsilon) \cdot \alpha,
	\end{align*}
	every~$\alpha$-approximation for $\lambda'$ is a $((1 + \varepsilon) \cdot \alpha)$-approximation for~$\lambda$. 
\end{enumerate}
These arguments result in a feasible adaption of the algorithm of~\cite{Papadimitriou+Yannakakis:multicrit-approx} for computing approximation sets: Given an instance of a $d$-objective minimization/maximization problem, $\ALG_\alpha$, and $0 < \varepsilon < 1$, compute~$\LB$, $\UB$, $\varepsilon'$, and $\beta$ and construct the grid~$\Grid$. Then, call~$\ALG_{\alpha}$ for each grid weight vector~$\lambda' \in \Grid$ and collect all solutions. By the argumentation outline so far, the computed set of solutions constitutes a $((1 + \varepsilon) \cdot \alpha)$-convex approximation set. This approach essentially coincides with the one provided in~\cite{Helfrich:mp-approximation}. The running time depends linearly on number of calls to $\ALG_{\alpha}$, which coincides with the cardinality of the grid~$\Grid$. The next result shows that the cardinality of~$\Grid$ is indeed bounded polynomially in the instance size and $\frac{1}{\varepsilon}$.
\begin{lemma}\label{lem:Approx+Convex:Cardinality-grid}
	For $\alpha \geq 1$ and $0 < \varepsilon < 1$, define $\Grid$ as in \eqref{eq:Approx+Convex:grid-finite}. Then,
	\begin{align*}
	\lvert \Grid \rvert \in \mathcal{O} \left( \left( \frac{1}{\varepsilon} \cdot \log \frac{1}{\varepsilon} + \frac{1}{\varepsilon} \cdot \log \frac{\UB}{\LB} + \frac{1}{\varepsilon} \cdot \log \alpha \right)^{d-1} \right).
	\end{align*}
\end{lemma}
\begin{proof}
Let, as above, $\varepsilon' = \sqrt{1 + \varepsilon} - 1$ and $\beta = (1 + \varepsilon') \cdot \alpha$. Note that
\begin{align*}
\frac{ \left( (1 + \varepsilon')^{a_1} , \ldots, (1 + \varepsilon')^{a_d} \right) }{\norm{\left( (1 + \varepsilon')^{a_1} , \ldots, (1 + \varepsilon')^{a_d} \right)}_1} =  \frac{ \left( 1, (1 + \varepsilon')^{a_2 - a_1} , \ldots, (1 + \varepsilon')^{a_d - a_1} \right) }{\norm{\left( 1, (1 + \varepsilon')^{a_2 - a_1} , \ldots, (1 + \varepsilon')^{a_d - a_1} \right)}_1}.
\end{align*}
Hence, we can assume that $a_1 = 0$, so the cardinality of $\Grid$ is bounded by
\begin{align*}
&\left( \lceil \log_{1 + \varepsilon'}\left( 1 - (d - 1) \cdot \lb \right ) + 1 \rceil - \lfloor \log_{1 + \varepsilon'} ( \lb ) \rfloor \right)^{d-1} \\
&\in \mathcal{O} \left( \log_{1 + \varepsilon'} \left( \frac{1 - (d-1) \cdot \lb}{\lb} \right)^{d-1} \right) \\
&= \mathcal{O} \left( \log_{1 + \varepsilon'} \left( \frac{1 - \frac{d-1}{d!} \frac{\varepsilon' \cdot \LB}{ \beta \cdot \UB}} {\frac{1}{d!} \frac{\varepsilon' \cdot \LB}{ \beta \cdot \UB}} \right)^{d-1} \right) \\
&= \mathcal{O} \left(  \log_{1 + \varepsilon'} \left( \frac{ \beta \cdot \UB}{\varepsilon' \cdot \LB} \right)^{d-1}\right) \\
&= \mathcal{O} \left( \left(  \frac{ \log \frac{1}{\sqrt{1 + \varepsilon}-1}}{ \frac{1}{2} \log (1 + \varepsilon)}  + \frac{ \log \frac{\UB}{\LB}}{ \frac{1}{2} \log (1 + \varepsilon)} + \frac{\log \alpha}{ \frac{1}{2} \log (1 + \varepsilon)}\right)^{d-1} \right) \\
&= \mathcal{O}  \left( \left( \frac{1}{\varepsilon} \cdot \log \frac{1}{\varepsilon} + \frac{1}{\varepsilon} \cdot \log \frac{\UB}{\LB} + \frac{1}{\varepsilon} \cdot \log \alpha \right)^{d-1} \right).
\end{align*}
Hereby, note that $$2^{\varepsilon} = 2^{(1 - \varepsilon) \cdot 0 + \varepsilon \cdot 1} \leq (1 - \varepsilon) \cdot 2^0 + \varepsilon \cdot 2^1 = 1 + \varepsilon$$
by convexity of exponential functions and $0 \leq \varepsilon \leq 1$, which implies that~$\varepsilon \leq \log(1 + \varepsilon)$.  
\end{proof}

\subsection{The Rounding Schemes}\label{sec:rounding-schemes}
Next, we propose two  rounding schemes for weight vectors inspired by the Continuity and the Projection Property in order to guarantee that~$\ALG_\alpha$ is called for grid weight vectors~$\lambda' \in \Grid$ only:
\begin{enumerate}
    \item Given $\delta > 0$ and a weight vector~$\lambda \in \Compact$, \texttt{GridRounding}{($\lambda$,$\delta$)} returns a weight vector~$\lambda' \in \Lambda$ such that, for~$\gamma \geq 1$, it holds that  $\lambda \in \Ball{\lambda'}{\gamma}{\delta}$. In particular, if $\delta \in \{\varepsilon, \varepsilon'\}$, it holds that $\lambda' \in \Grid$. 
    
    \item Given a weight vector~$\lambda \in \Lambda$, $\varepsilon'$, $\beta$, $\LB$, and $\UB$, \texttt{BoundaryRounding}($\lambda$, $\varepsilon'$, $\beta$, $\LB$, $\UB$) returns a weight vector~$\lambda' \in \Compact$ and a boolean~$\flag$. If $\flag = \text{FALSE}$, it holds that $\lambda' = \lambda$, so $\lambda \in \Compact$ and \texttt{GridRounding}{($\lambda$,$\varepsilon$)} is applied afterwards. If $\flag = \text{TRUE}$, it holds that $\lambda \in \Lambda \setminus \Compact$ and $\lambda' \in \Compact$ such that every~$\beta$-approximation for $\lambda'$ is a $(\beta + \varepsilon')$-approximation for $\lambda$. In this case, \texttt{GridRounding}{($\lambda$,$\varepsilon'$)} is applied afterwards.
\end{enumerate}
Applying both rounding schemes before calling~$\ALG_\alpha$ guarantees that the total number of calls to~$\ALG_\alpha$ is polynomially in the instance size and $\frac{1}{\varepsilon}$, which lays the foundation for obtaining a running time polynomial in the instance size and $\frac{1}{\varepsilon}$ given that~$\ALG_\alpha$ runs in polynomial time. 

\subsubsection{Grid Rounding}
Given $\delta > 0$ and a weight vector~$\lambda \in \Compact$,  the \texttt{GridRounding} scheme returns a weight vector~$\lambda' \in \Lambda$ such that, for $\gamma \geq 1$, it holds that $\lambda \in \Ball{\lambda'}{\gamma}{\delta}$. 
This is simply done by determining~$\bar{\lambda}_i \coloneqq \lceil \log_{1 + \delta} (\lambda_i) \rceil$ for $i = 1,\ldots, d$ and determining~$\lambda'$ as the projection of~$\bar{\lambda}$ onto the weight set. Algorithm~\ref{alg:Approx+Convex:GridRounding} summarizes this.
\begin{algorithm}[tb]
\small
	\begin{algorithmic}[1]
	
	\Require{A weight vector $\lambda \in \Compact$ and $\delta > 0$.}
	
	\Ensure{A weight vector $\lambda' \in \Grid$ such that, for $\gamma \geq 1$, every $\gamma$-approximation for $\lambda'$ is a $((1 + \delta) \cdot \gamma)$-approximation for $\lambda$.}
	
	\For{$i=1,\ldots, d$}
		\State$\bar{\lambda}_i \leftarrow (1 + \delta)^{\lceil \log_{1 + \delta}(\lambda_i) \rceil }$;\label{alg:Approx-Convex:GridRounding:Bisection}
	\EndFor
 
	\State $\lambda' \leftarrow \frac{\bar{\lambda}}{\norm{\bar{\lambda}}_1}$;
 
	\State \Return  $\lambda'$;
 \end{algorithmic}
	\caption{Rounding scheme~$\texttt{GridRounding}(\lambda,\delta)$.}
	\label{alg:Approx+Convex:GridRounding}
\end{algorithm}
\begin{proposition}\label{prop:Approx+Convex:GridRoundingCorrect+RunTime}
	Algorithm~\ref{alg:Approx+Convex:GridRounding} returns a weight vector~$\lambda' \in \Compact$ such that, for $\gamma \geq 1$, every $\gamma$-approximation for $\lambda'$ is a $(1 + \delta)\cdot\gamma$-approximation for $\lambda$ in time
	\begin{align*}
	\mathcal{O}\left( \log \left( \frac{1}{ \log(1 + \delta) \cdot \lb} \right)\right),
	\end{align*}
	where~$\lb$ is defined as in~\eqref{eq:Approx+Convex:lb}.
\end{proposition}
\begin{proof}
    Let $\lambda \in \Compact$ and $\delta > 0$. Then, Algorithm~\ref{alg:Approx+Convex:GridRounding} computes
    the weight vector~$\bar{\lambda} \coloneqq (1 + \delta)^{\lceil \log_{1 + \delta}(\lambda_i) \rceil}$ in Step~\ref{alg:Approx-Convex:GridRounding:Bisection}. Since
    \begin{align*}
    \frac{1}{1 + \delta} \cdot \bar{\lambda}_i = (1 + \delta)^{\lceil \log_{1 + \delta}(\lambda_i) \rceil - 1} &\leq (1 + \delta)^{\log_{1 + \delta}(\lambda_i)} = \lambda_i \\
    &\leq (1 + \delta)^{\lceil \log_{1 + \delta}(\lambda_i) \rceil} = \bar{\lambda}_i,
    \end{align*}
    it holds that $\lambda \in \Ball{\frac{\bar{\lambda}}{\norm{\bar{\lambda}}}}{\gamma}{\delta}$. Thus, every $\gamma$-approximation for $\lambda' = \frac{\bar{\lambda}}{\norm{\bar{\lambda}}}$ is a $((1+ \delta)\cdot \gamma)$-approximation for~$\lambda$ by the Continuity Property~\ref{property:Approx+Convex:Continuity}.
    
    Moreover, the bound~$\lb$ given in Lemma~\ref{lem:Approx+Convex:bounds-on-weight-set} allows to bound the asymptotic worst-case running time:
    Lemma~\ref{lem:Approx+Convex:bounds-on-weight-set} yields that~$\lambda_i \geq \lb$, which implies that~$\lambda_i \leq 1 - (d-1) \cdot \lb$ holds as well. Then, Step~\ref{alg:Approx-Convex:GridRounding:Bisection} is equivalent to finding, for each~$i=1,\ldots,d$, an integer
	\begin{align*}
	a_i \in \left\{ \lceil \log_{1 + \delta}(\lb) \rceil, \lceil \log_{1 + \delta}(\lb) \rceil + 1, \ldots, \lceil \log_{1+ \delta}( 1 - (d-1) \cdot \lb)  \rceil \right\}
	\end{align*}
	such that $(1 + \delta)^{a_i - 1} \leq \lambda_i \leq (1 + \delta)^{a_i}$. For the base-two logarithm and $t > 1$, it holds true that~$ \frac{3}{4} \cdot \left( 1 - \frac{1}{t} \right) < \log(t) < 2 \cdot ( t - 1)$. Since~$\lb < 1$, this implies that
	\begin{align*}
		\log_{1 + \delta}(\lb) =  - \frac{\log(\frac{1}{\lb})}{\log(1 + \delta)} > - \frac{2 \cdot (\frac{1}{\lb} - 1)}{\log(1 + \delta)}
	\end{align*}
	and, since~$1 - (d-1) \cdot \lb < 1$ as well,
	\begin{align*}
		\log_{1 + \delta} ( 1 - (d-1) \cdot \lb) &= - \frac{\log(\frac{1}{1 - (d-1) \cdot \lb})}{\log(1 + \delta)}\\
		 &< - \frac{\frac{3}{4} \cdot (1 + (1 - (d-1) \cdot \lb))}{\log(1 + \delta)} \\
		&= \frac{3}{4} \cdot \lb \cdot  \frac{d - 1}{\log(1 + \delta)} \leq \frac{d - 1}{\log(1 + \delta)}
	\end{align*}
	Hence, each integer~$a_i$ can be found by a bisection search in time
	\begin{align*}
		&\mathcal{O}\left( \log\left( \lceil \log_{1+ \delta}( 1 - (d-1) \cdot \lb)  \rceil -  \lceil \log_{1 + \delta}(\lb) \rceil \right)   \right) \\
		= &\mathcal{O}\left(\log\left(  \frac{d - 1}{\log(1 + \delta)}  + \frac{2 \cdot (\frac{1}{\lb} - 1)}{\log(1 + \delta)}  \right) \right) = \mathcal{O}\left( \log\left( \frac{1}{\log(1 + \delta)\cdot \lb } \right) \right),
	\end{align*}
	which concludes the proof.
\end{proof}
When choosing $\delta \in \{\varepsilon, \varepsilon'\}$ with $ \varepsilon' = \sqrt{1 + \varepsilon} - 1$, it can be proven similar to the proof of Lemma~\ref{lem:Approx+Convex:Cardinality-grid} that the running time of Algorithm~\ref{alg:Approx+Convex:GridRounding} resolves to
\begin{align*}
\mathcal{O} \left( \log \left( \left( \frac{1}{\varepsilon} \cdot \log \frac{1}{\varepsilon} + \frac{1}{\varepsilon} \cdot \log \frac{\UB}{\LB} + \frac{1}{\varepsilon} \cdot \log \alpha \right)^{d-1} \right) \right) = \mathcal{O} ( \log (\lvert \Grid \rvert ) ).
\end{align*}
Note that the weight vector~$\lambda$ is part of the input. Thus, the running time of Algorithm is polynomial in the encoding length of the instance \emph{and} the encoding length of the weight vector~$\lambda$.

\subsubsection{Boundary Rounding}

Given a weight vector~$\lambda \in \Lambda$, $\varepsilon'$, $\beta$, $\LB$, and $\UB$, the \texttt{BoundaryRounding} scheme returns a weight vector~$\lambda' \in \Compact$ and a boolean~$\flag$. If $\flag = \text{FALSE}$, it holds that $\lambda' = \lambda$ and, thus, $\lambda \in \Compact$. If $\flag = \text{TRUE}$, it holds that $\lambda \in \Lambda \setminus \Compact$ and $\lambda' \in \Compact$ such that every~$\beta$-approximation for $\lambda'$ is a $(\beta + \varepsilon')$-approximation for $\lambda$. 

In order to understand this rounding scheme, we need several auxiliary results, most of which have again been proven in~\cite{Helfrich:mp-approximation}. In the following, it will often be convenient to consider non-normalized weight vectors~$\lambda \in \R^d_\geqq \setminus \{0\}$. Nevertheless, rescaling of a weight vector~$\lambda$ with the positive scalar~$\sum_{i=1}^d \lambda_i > 0$ allows to reconnect all results to the weight set~$\Lambda$. The first result concerns convexity and approximation: 
\begin{lemma}[\cite{Helfrich:mp-approximation}, Lemma~4.2]\label{lem:Approx+Convex:convexity} 
	Let~$\gamma' \geq 1$ and $\Lambda' \subseteq \R^d_\geqq \setminus \{0\}$. Then, every $\gamma'$-approximation for $\Lambda'$ is a $\gamma'$-approximation for $\conv(\Lambda')$. 
\end{lemma}
The second auxiliary result explains the structure of the set~$\Compact$. Let~$\lambda$ be a weight vector whose components~$\lambda_i$ for~$i$ in some index set $\emptyset \neq I \subsetneq \{1,\ldots, d\}$ sum up to a small threshold. Further, the projection $\proj^I:\R^{d} \rightarrow \R^{d}$ that maps all components~$\lambda_i$ of a vector~$\lambda \in \R^{d}$ with indices~$i\in I$ to zero is defined by
\begin{align*}
\proj^{I}_i (\lambda) \coloneqq \begin{cases}
0, &\text { if } i \in I,\\
\lambda_i, &\text { else}.\end{cases} 
\end{align*}
Then, every approximate solution for~$\lambda$ is still an approximate solution for~$\proj^I(\lambda)$ with a \enquote*{sufficiently good} approximation guarantee. Together with Lemma~\ref{lem:Approx+Convex:convexity}, we obtain that every approximate solution for~$\lambda$ is still an approximate solution for~$\conv(\{\lambda, \proj^I(\lambda) \})$ with a \enquote*{sufficiently good} approximation guarantee. This is formally capture in the following statement.
\begin{lemma}[\cite{Helfrich:mp-approximation}, Lemma~4.3]\label{lem:ApproximationGuaranteeWithinTreshhold}
	Let $0 < \varepsilon' < 1$ and $\beta \geq 1$. Further, let $\emptyset \neq I \subsetneq \{1, \ldots, d\}$ be an index set and let~$\lambda \in \Lambda$ be a weight vector for which
	\begin{align*}
	\sum_{i \in I} \lambda_i = \frac{\varepsilon' \cdot \LB}{\beta \cdot \UB} \cdot \min_{j \notin I} \lambda_j.
	\end{align*}
	Then, every $\beta$-approximation for $\lambda$ is a $(\beta + \varepsilon')$-approximation for $\conv(\{\lambda, \proj^I(\lambda)  \} )$. 
\end{lemma}
\noindent This suggests to define for $0 < \varepsilon' < 1$ and $\beta \geq 1$~\citep{Helfrich:mp-approximation} 
\begin{align}\label{eq:c}
	c \coloneqq \frac{\varepsilon' \cdot \LB}{\beta \cdot \UB} \in (0,1)
\end{align}
and, for each index set $\emptyset \neq I \subsetneq \{1,\ldots,d\}$,
\begin{align}\label{eq:Psets}
P_{<} (I) \coloneqq \left\{ \lambda \in \R^{d}_\geqq \st \sum_{i \in I} \lambda_i < c \cdot \lambda_j \text{ for all } j \notin I \right\}.
\end{align}
The sets~$P_\leq (I)$ and $P_= (I)$ are defined analogously by replacing \enquote{$<$} by \enquote{$\leq$} and \enquote{$=$}, respectively. Note that the sets~$P_{<}(I), P_{=}(I)$, and $P_\leq(I)$ are defined in the superset~$\R^d_\geq$ of $\Lambda$. 
Let~$\lambda \in P_<(I)$ for some index set $\emptyset \neq I \subsetneq \{1,\dots, d\}$. The next lemma shows that a weight vector~$\bar{\lambda}$ such that we can apply Lemma~\ref{lem:ApproximationGuaranteeWithinTreshhold} can be constructed easily:
\begin{lemma}[\cite{Helfrich:mp-approximation}, Lemma~4.5]\label{lem:boundaryRounding}
	Let $\emptyset \neq I \subsetneq \{1,\dots, d\}$ be a nonempty index set and let~$\lambda \in P_<(I)$. Define
	\begin{align}\label{eq:BoundaryRounding}
	\bar{\lambda}_i \coloneqq 
	\begin{cases}
	\frac{\lambda_i}{\sum_{j \in I} \lambda_j} \cdot c \cdot \min_{j \notin I} \lambda_j, &\text{ if } i \in I \text{ and } \sum_{j \in I} \lambda_j >0,\\
	\frac{1}{\lvert I \rvert} \cdot c \cdot \min_{j \notin I} \lambda_j, &\text{ if } i \in I \text{ and } \lambda_j =0 \text{ for all } j \in I,\\
	\lambda_i, &\text{ if } i \notin I.
	\end{cases}
	\end{align}
	Then, $\bar{\lambda} \in P_=(I)$ and $\lambda \in \conv(\{\bar{\lambda},\proj^I(\bar{\lambda})\})$. 
\end{lemma}
Hence, when given a weight vector~$\lambda \in P_<(I)$ for some index set~$I$, it can be \enquote{rounded} to a weight vector~$\bar{\lambda} \in P_=(I)$ using~\eqref{eq:BoundaryRounding}. A $\beta$-approximation for$~\bar{\lambda}$ is then a $(\beta + \varepsilon')$-approximation for~$\lambda$ due to Lemma~\ref{lem:Approx+Convex:convexity} and Lemma~\ref{lem:ApproximationGuaranteeWithinTreshhold}. Hereby, note that weight vectors can be contained in $P_<(I) \cap P_<(I')$ of two (or more) different index sets~$I$ and~$I'$ and Construction~\eqref{eq:BoundaryRounding} for~$I$ might result in a weight vector that is still contained in~$P_<(I')$. Upon that, applying Construction~\eqref{eq:BoundaryRounding} for~$I'$ might result in a weight vector~$\lambda'$ that indeed satisfies~$\lambda' \notin P_<(I) \cup P_<(I')$, but does \emph{not necessarily} satisfy~$\lambda' \in P_=(I)$ anymore, see Figure~\ref{fig:IllustrationRounding}. So, we would \enquote{lose} applicability of Lemma~\ref{lem:ApproximationGuaranteeWithinTreshhold} for~$I$, which would impede correctness.
Nevertheless, Algorithm~\ref{alg:BoundaryRounding} inductively applies this rounding idea \emph{in a particular order} to obtain a weight vector~$\bar{\lambda} \in \Compact$ after at most~$d$ rounding steps such that the Projection Property holds: every $\beta$-approximation for$~\bar{\lambda}$ is a $(\beta + \varepsilon')$-approximation for~$\lambda$. To see that such an order exists, the following observation is crucial:
\begin{observation}\label{obs:sorting-P_<(I)}
	If $\lambda \in \R^d_\geqq$ such that $\lambda_1 \leq \lambda_2 \leq \ldots \leq \lambda_d$ and $\lambda \in P_<(I)$ for some index set $\emptyset \neq I \subsetneq \{1,\dots,d\}$, then $I = \{1,\ldots, \lvert I \rvert\}$.
\end{observation}
Consequently, we can check if a weight vector~$\lambda$ is contained in~$P_<(I)$ for some index set~$I \subsetneq \{1,\dots,d\}$ and, if yes, we can determine all such index sets, by identifying a permutation~$\sigma$ of $\{1,\ldots,d\}$ such that $\lambda_{\sigma(1)} \leq \lambda_{\sigma(2)} \leq \ldots \leq \lambda_{\sigma(d)}$ holds true and checking, for~$k=1, \ldots,d-1$, whether~$\sum_{i=1}^k \lambda_{\sigma(i)} < c \cdot  \lambda_{\sigma(k+1)}$ holds. It can be shown that the \enquote{rounding}~\eqref{eq:BoundaryRounding} of a weight vector~$\lambda$ with $\lambda_1 \leq \lambda_2 \leq \ldots \leq \lambda_d$ according to some index set $\{1, \ldots, k\}$ preserves the sorting of the components and whether the weight vector is not contained in some $P_{<}(\{1, \ldots, k'\})$ for every~$1 \leq k' < k$, see Appendix~\ref{lem:P_<(I)-rounding-invariant}.

\medskip

So, in summary, \texttt{BoundaryRounding} works as follows: Given a weight vector~$\lambda \in \R^d_\geqq$, first find a permutation~$\sigma$ of $\{1,\ldots, d\}$ such that $\lambda_{\sigma(1)} \leq \lambda_{\sigma(2)} \leq \ldots \leq \lambda_{\sigma(d)}$. Then, for each $k = 1 ,\ldots,d-1$, check whether $\sum_{i=1}^k \lambda_{\sigma(i)} < c \cdot \lambda_{\sigma(k+1)}$ holds. If yes, apply the \enquote{rounding}~\eqref{eq:BoundaryRounding} with index set~$\{\sigma(1),\ldots,\sigma(k)\}$ and continue with the updated weight vector and~$k+1$. Otherwise, continue immediately with the next iteration~$k+1$.
At the end, it is left to normalize the weight vector to obtain a weight vector~$\lambda' \in \Lambda$ such that $\lambda' \notin P_<(I)$ for every index set $I \subseteq \{1,\ldots, d\}$. Let $k^1,\ldots k^L$, $1 \leq L \leq d-1$ be the iterations in which a rounding has been applied, and set~$I^\ell \coloneqq \{\sigma(1),\ldots\sigma( k^\ell)\}$. Then, it is guaranteed that the weight vector~$\lambda'$ is contained in $P_=(I^\ell)$ for $\ell = 1, \ldots, L$ and, in particular, $\lambda' \in \conv\left( \left\{\lambda', \proj^{I^1}(\lambda') , \ldots,\proj^{I^L}(\bar{\lambda}) \right\}\right)$. Consequently, applying Lemma~\ref{lem:ApproximationGuaranteeWithinTreshhold} and Lemma~\ref{lem:Approx+Convex:convexity} yields that every~$\beta$-approximation for~$\lambda' \in \Compact$ is a $(\beta + \varepsilon')$-approximation for $\lambda$. Figure~\ref{fig:IllustrationRounding} illustrates this procedure, and Algorithm~\ref{alg:BoundaryRounding} summarizes it.
Note that, if no rounding has been applied at all during these steps, then~$\lambda' = \lambda \in \Compact$. Thus, we can check whether $\lambda \in \Compact$ by applying \texttt{BoundaryRounding}  as well. The next result states correctness as well as an asymptotic worst-case running time analysis of Algorithm~\ref{alg:BoundaryRounding}. Since the proof is rather technical, it is stated in Appendix~\ref{proof-prop-correctness-boundary}.
\begin{algorithm}[tb]
\small
\begin{multicols}{2}
 \begin{algorithmic}[1]
	\Require{A weight vector $\lambda \in \Lambda$, $\beta \geq 1$, and $0 <\varepsilon' <1$, lower and upper bounds $\LB$ and $\UB$ such that $f_i(x) \in \{0\} \cup [\LB, \UB]$ for all $i=1, \dots, d$ and all $x \in X$.}
	
	\Ensure{A weight vector $\lambda' \in \Compact$ such that every $\beta$-approximation for $\bar{\lambda}$ is a $(\beta + \varepsilon')$-approximation for $\lambda$. A boolean $\flag$ that indicates whether a rounding has been applied.}

	\State $\sigma \leftarrow$ permutation such that 
 
    $\lambda_{\sigma(1)} \leq \lambda_{\sigma(2)}\leq \ldots \leq \lambda_{\sigma(d)}$; \label{alg:BoundaryRounding:Sigma1}
	
	\State $\bar{\lambda}_i \leftarrow \lambda_{\sigma(i)}$ for $i =1, \dots, d$;\label{alg:BoundaryRounding:Sigma2}
	
	\State $c \leftarrow \frac{\varepsilon' \cdot \LB} {\beta \cdot \UB} \in (0,1)$;
	
	\State $\flag \leftarrow \textup{\text{FALSE}}$;
	
	\For{$k = 1,\ldots,d-1$\label{alg:BoundaryRounding:BeginForLoop}}
		\If{$\sum_{j=1}^k \bar{\lambda}_j < c \cdot \bar{\lambda}_{k+1}$\label{alg:BoundaryRounding:IfConditionPI}}
		      \State $\flag \leftarrow \textup{TRUE}$;
			
			\State $a \leftarrow \sum_{i = 1}^k \bar{\lambda}_i$;
			
			\For{$i=1, \ldots, k$}				
				\If{$a >0$}
					\State $\bar{\lambda}_i \leftarrow \frac{\bar{\lambda}_i}{a} \cdot  c \cdot \bar{\lambda}_{k+1}$;
				\Else
					\State $\bar{\lambda}_i \leftarrow \frac{c}{k+1} \cdot \bar{\lambda}_{k+1}$;
     \EndIf
			\EndFor	
		\EndIf
		\label{alg:BoundaryRounding:EndForLoop}
    \EndFor
	\State $\sigma^{-1} \leftarrow $ Inverse permutation of $\sigma$;
	
	\State $\tilde{\lambda}_i \leftarrow \bar{\lambda}_{\sigma^{-1}(i)}$ for $i = 1, \ldots, d$;\label{alg:BoundaryRounding:InverseSorting}
	
	\State $\lambda' \leftarrow \frac{\tilde{\lambda}}{\norm{\tilde{\lambda}}_1}$;
	
	\State \Return $\lambda'$, $\flag$
 \end{algorithmic} \end{multicols}
	\caption{Rounding scheme~\texttt{BoundaryRounding}($\lambda$, $\varepsilon'$, $\beta$, $\LB$, $\UB$).}
	\label{alg:BoundaryRounding}
\end{algorithm}

\begin{figure}[tb]
	\centering
    \includegraphics[width = 0.8\textwidth]{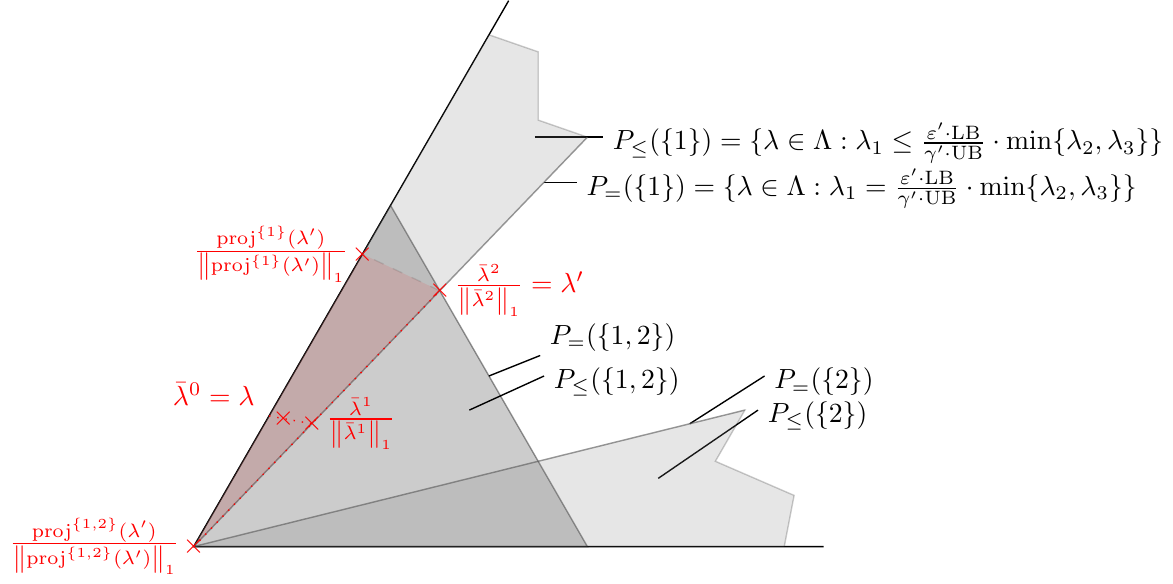}
	\caption{\citep{Helfrich:mp-approximation} Illustration of the sequence of weight vectors constructed in Algorithm~\ref{alg:BoundaryRounding} for a weight vector $\lambda \in P_<(\{1\}) \cup P_<(\{1,2\}$ with $\lambda_1 < \lambda_2 < \lambda_3$. First, the weight vector~$\bar{\lambda}^0 = \lambda$ is rounded to $\bar{\lambda}^1$ by applying~\eqref{eq:BoundaryRounding} with index set~$\{1\}$. Then, the weight vector~$\bar{\lambda}^1$ is rounded to $\bar{\lambda}^2$ by applying~\eqref{eq:BoundaryRounding} with index set~$\{1,2\}$. For the purpose of a concise illustration, all weight vectors are normalized, i.e., their components sum up to one. Note that applying initially a rounding with index set~$\{1,2\}$ and then a rounding with~$\{1\}$ would yield a weight vector~$\tilde{\lambda} \notin P_=(\{1,2\})$. Consequently, Lemma~\ref{lem:ApproximationGuaranteeWithinTreshhold} could not be applied with $I = \{1,2\}$.}
	\label{fig:IllustrationRounding}
\end{figure}
\begin{proposition}\label{prop:CorrectnessBoundaryRounding}
	Let $\beta \geq 1$, $0 < \varepsilon' < 1$, and let~$\lambda \in \Lambda$. 
	Then, Algorithm~\ref{alg:BoundaryRounding} returns a weight vector $\lambda' \in \Compact$ such that, if $\flag = \text{TRUE}$, every $\beta$-approximation for $\lambda'$ is a $(\beta + \varepsilon')$-approximation for $\lambda$. Otherwise, if $\flag = \text{\text{FALSE}}$, it holds that $\lambda' = \lambda$ and, thus, every $\beta$-approximation for $\lambda'$ is a $\beta$-approximation for $\lambda$. Further, Algorithm~\ref{alg:BoundaryRounding} has worst-case running time~$\mathcal{O}(1)$.
\end{proposition}
Note that the input of Algorithm~\ref{alg:BoundaryRounding} includes the weight vector~$\lambda$. Thus, the running time depends on the encoding length of the instance \emph{and} the encoding length of the weight vector~$\lambda$.
\section{An Approximate Dual Variant of Benson's Outer Approximation Algorithm}\label{sec:Approx+Convex:oaa-algorithm}
We now present the details of our algorithm for computing convex approximation sets.
As already outlined in Section~\ref{sec:Foundations}, our algorithm essentially combines the idea of the dual variant of Benson's Outer Approximation Algorithm with the grid approach presented in~\cite{Helfrich:mp-approximation}. Given~$0 < \varepsilon < 1$, it uses an $\alpha$-approximation algorithm~$\ALG_{\alpha}$ for the weighted sum scalarization to construct a $((1 + \varepsilon) \cdot \alpha)$-convex approximation set while relying on a polynomial (in the instance size and $\frac{1}{\varepsilon}$) number of calls to~$\ALG_\alpha$. Again, if a polynomial-time exact algorithm or an (F)PTAS for the weighted sum scalarization is available, our algorithm yields an M(F)PTcAS. In contrast to the algorithms of~\cite{Diakonikolas:thesis} and \cite{Helfrich:mp-approximation}, the choice of weight vectors for which~$\ALG_{\alpha}$ is called is done adaptively and is based on information provided by already computed solutions. This causes the asymptotic worst-case running time to be significantly worse (but still polynomial). However, as shown in the computational study provided in Section~\ref{sec:CompStudy}, our algorithm outperforms the existing ones in practical running time, and the cardinality of the returned convex approximation set is significantly smaller.

\medskip

The first result guarantees termination similar to Proposition~\ref{prop:termination-DualBenson} in the exact case. It states that, for $\beta \geq 1$, it is sufficient to investigate exclusively extreme points of~$D(S)$ in order to determine whether $S$ is a $\beta$-convex approximation set. 
\begin{proposition}\label{prop:SuffCondForApproxPoly}
	Let $\beta \geq 1$ and $S \subseteq X$ be a finite set of feasible solutions. If every extreme point $(\lambda_1,\ldots, \lambda_d, z)$ of $D(S)$ satisfies $z \leq \beta \cdot \lambda^\top f(x)$ for all $x \in X$ in the case of minimization, and $z \geq \frac{1}{\beta} \cdot \lambda^\top f(x)$ for all $x \in X$ in the case of maximization, then~$S$ is a $\beta$-convex approximation set. 
\end{proposition}
\begin{proof}
	Let~$(\lambda^*_1,\ldots,\lambda^*_d,z^*)$ be an extreme point of $D(S)$. Then, by construction of~$D(S)$, the point~$(\lambda^*_1,\ldots,\lambda^*_d,z^*)$ satisfies at least $d + 1$ of the following equations:
	\begin{align*}
		 \lambda_i &= 0,\; i =1,\ldots,d, \\
		 \sum_{i = 1}^d \lambda_i &= 1,\\
		 (\lambda_1,\ldots, \lambda_d)^\top f(\hat{x}) &= z,\; \hat{x} \in S.
	\end{align*}
	Since no weight vector~$\lambda \in \Lambda$ can satisfy $\lambda_i = 0$ for all $i \in\{1,\ldots d\}$ and $\sum_{i = 1}^d \lambda_i = 1$ simultaneously, there exists a solution~$x^* \in S$ such that $(\lambda^*_1,\ldots, \lambda^*_d)^\top f(x^*) = z^*$. By assumption, this means that $x^*$ is a $\beta$-approximation for $(\lambda^*_1,\ldots, \lambda^*_d)$. 
	
	\medskip
	
	We prove the claim for the minimization case. The maximization case can be handled similarly. Let~$\lambda \in \Lambda$ and set $x' \coloneqq \arg \min_{\hat{x}\in S} \lambda^\top f(\hat{x})$. Then, $(\lambda_1,\ldots,\lambda_d,\lambda^\top f(x')) \in D(S)$  and there exist extreme points~$(\lambda^1_1,\ldots,\lambda^1_d,z^1)$, $\ldots,(\lambda^L_1,\ldots,\lambda^L_d,z^L)$ of $D(S)$, scalars~$\theta_1,\ldots \theta_L \geq 0$ with $\sum_{\ell = 1}^L \theta_\ell = 1$, and a scalar~$t \geq 0$ such that
	\begin{align*}
		(\lambda,\lambda^\top f(x')) = \sum_{\ell = 1}^L \theta_\ell (\lambda^\ell_1,\ldots,\lambda^\ell_d,z^\ell) - t \cdot e^{d+1},
	\end{align*}
	where~$e^{d+1}$ denotes the $(d+1)$th unit vector in $\R^{d+1}$. Thus, as shown above, there exist solutions~$x^1,\ldots x^L \in S$ such that $z^\ell = (\lambda^\ell)^\top f(x^\ell)$ for $\ell = 1,\ldots L$. This implies that, for every~$x \in X$,
	\begin{align*}
	\lambda^\top f(x') &= \sum_{\ell = 1}^L \theta_\ell \cdot (\lambda^\ell)^\top f(x^\ell) - t \leq \sum_{\ell = 1}^L \theta_\ell \cdot (\lambda^\ell)^\top f(x^\ell) \\
	&\leq \sum_{\ell = 1}^L \theta_\ell \cdot \beta \cdot (\lambda^\ell)^\top f(x) = \beta \cdot \left(\sum_{\ell = 1}^L \theta_\ell \lambda^\ell\right)^\top f(x) = \beta \cdot \lambda^\top f(x).
	\end{align*}
\end{proof}
Our algorithm now works as follows: For a given instance~$\I$ of a $d$-objective minimization/maximization problem, a scalar $0 < \varepsilon < 1$, and an $\alpha$-approximation algorithm~$\ALG_\alpha$ for the weighted sum scalarization, it first computes $\LB$ and $\UB$ as well as $\varepsilon' = \sqrt{1 + \varepsilon} - 1$ and $\beta = (1 + \varepsilon') \cdot \alpha$. In the next step, it calls $\ALG_\alpha$ for the weight vector~$\lambda = (\frac{1}{d}, \ldots, \frac{1}{d}) \in \Grid$. Based on the returned solution~$x$, it initializes~$S = \{x\}$, constructs the polyhedron~$D \coloneqq D(S)$, and initializes a set of already investigated weight vectors~$R = \{\lambda\}$. The algorithm maintains a list~$M$ of extreme points of~$D$ that must be investigated, which is initialized as all extreme points of~$D$. Then, in each iteration, an extreme point~$(\lambda^*_1,\ldots, \lambda^*_d,z^*)$ in $M$ is picked and boundary rounding with~$\varepsilon'$ and~$\beta$ is applied to $\lambda^* \coloneqq (\lambda^*_1, \ldots, \lambda^*_d)$. Depending on whether the resulting weight vector~$\lambda' \in  \Compact$ is indeed a rounded one or not, the grid rounding procedure with~$\varepsilon'$ or~$\varepsilon$, respectively, is applied to obtain a weight vector~$\lambda'' \in \Grid$. If~$\lambda''$ is contained in~$R$, $\ALG_\alpha$ has already been applied for~$\lambda''$ and an $\alpha$-approximation~$x''$ for~$\lambda''$ has already been found. As outlined in Section~\ref{sec:Foundations}, it is guaranteed that the solution~$x''$ is a $((1+ \varepsilon) \cdot \alpha)$-approximation for~$\lambda^*$ and, therefore, $z^* \leq (1 + \varepsilon) \cdot \alpha \cdot (\lambda^*)^\top f(x')$ for all $x' \in X$ in the case of minimization, and $z^* \geq \frac{1}{(1 + \varepsilon) \cdot \alpha} \cdot (\lambda^*)^\top f(x')$ for all $x' \in X$ in the case of maximization. Consequently, the extreme point $(\lambda^*_1,\ldots,\lambda^*_d,z^*)$ can be skipped. If~$\lambda''$ is not contained in~$R$, algorithm~$\ALG_\alpha$ is called for~$\lambda''$, the weight vector~$\lambda''$ is added to $R$, $S$ is updated to $S \cup \{x''\}$, and~$D$ is updated to $D \cap H(x'')$ with the obtained solution~$x''$. Then, the list~$M$ is reset to the list of all extreme points of the updated polyhedron~$D$ and the algorithm proceeds with the next iteration. This is repeated until
all rounded weight vectors~$\lambda''$ of all extreme points $(\lambda^*_1,\ldots, \lambda^*_d,z^*)$ of~$D$ are contained in~$R$. By Proposition~\ref{prop:SuffCondForApproxPoly}, this confirms that the set~$S$ containing all computed solutions is a $((1+\varepsilon) \cdot \alpha)$-convex approximation set.

Note that all extreme points of the updated polyhedron~$D$ must indeed be investigated since, due to the approximate nature of the solutions, it cannot be guaranteed that already investigated extreme points are also extreme points of the updated polyhedron. Moreover, Lemma~\ref{lem:Approx+Convex:Cardinality-grid} states a polynomial upper bound of the maximum number of obtained solutions. This guarantees that the vertex enumeration and checking whether a weight vector is contained in~$R$ runs in polynomial time. Algorithm~\ref{alg:FPTOAA} summarizes this procedure. In the remainder of this section, we formally prove its correctness and analyze its asymptotic worst-case running time.

 \begin{algorithm}[tb]
    \small
\begin{multicols}{2}    
 \begin{algorithmic}[1]
	\Require{An instance~$\I$ of a $d$-objective minimization/maximization problem, $0 < \varepsilon < 1$, an $\alpha$-approximation algorithm $\ALG_{\alpha}$ for the weighted sum scalarization.}
	
	\Ensure{A $((1 + \varepsilon) \cdot \alpha)$-convex approximation set for $\I$.}
	
	\State Compute $\LB$ and $\UB$;
	
	\State $\varepsilon' \leftarrow \sqrt{1 + \varepsilon} -1$
	
	\State $\beta \leftarrow (1 + \varepsilon') \cdot \alpha$;
	
	\State $\lambda \leftarrow (\frac{1}{d}, \ldots, \frac{1}{d})$;
	
	\State $x \leftarrow \ALG_{\alpha}(\lambda)$;
	
	\State $S \leftarrow \{x\}$;
	
	\State $D \leftarrow D(S)$;
	
	\State $M \leftarrow $ extreme points of $D$\label{alg:FPTOAA:vertexenumeration};
	
	\State $R \leftarrow \{\lambda\}$;
	
	\While{$M \neq \emptyset$}{
		\State Pick $(\lambda^*_1,\ldots, \lambda^*_d, z^*) \in M$;
  
        \State $M \leftarrow M \setminus \{(\lambda^*_1,\ldots, \lambda^*_d, z^*)\}$;
		
		\State $\lambda^* = (\lambda^*_1,\ldots, \lambda^*_d)$;
		
		\State $\lambda^*, \flag \leftarrow \texttt{BoundaryRounding}$\\  \hspace{2.4cm}$(\lambda^*,\varepsilon',\beta, \LB,\UB)$;
		
		\If{$\flag$ is $\text{TRUE}$}\State  $\lambda^* \leftarrow \texttt{GridRounding}(\lambda^*,\varepsilon')$;
		
		\Else
        \State  $\lambda^* \leftarrow \texttt{GridRounding}(\lambda^*,\varepsilon)$
        \EndIf
		\If{$\lambda^* \notin R$\label{alg:FPTOAA:conditionInL}}
			\State $x^* \leftarrow \ALG_{\alpha}(\lambda^*)$; 
   
            \State $S \leftarrow S \cup \{x^*\}$;
			
			\State $R \leftarrow R \cup \{\lambda^*\}$;\label{alg:FPTOAA:insertingInL}
			
			\State $D \leftarrow D(S)$;
			
			\State $M \leftarrow $ extreme points of $D$;\label{alg:FPTOAA:vertex-enum2}
    \EndIf
	}
    \EndWhile
	\State  \Return $S$;
 \end{algorithmic}
 \end{multicols}
	\caption{Approximate variant of the dual variant of Benson's Outer Approximation Algorithm.}
	\label{alg:FPTOAA}
\end{algorithm}

Since $\ALG_\alpha$ is applied for grid weight vectors $\lambda' \in \Grid$ only, Lemma~\ref{lem:Approx+Convex:Cardinality-grid}  implies that $\vert S \rvert \leq \lvert \Grid \rvert$ holds in every iteration of Algorithm~\ref{alg:FPTOAA}. Hence, the polyhedron~$D$ is the intersection of at most $\lvert \Grid \rvert + d + 2$ half-spaces (the half-spaces~$H(x)$ for all solutions $x \in S$, $\sum_{i=1}^d \lambda_i = 1$, and $\lambda_i \geq 0$ for $i = 1, \dots, d$). The consequences of this observation are twofold: First, we can apply Seidel's asymptotic upper bound theorem~\cite{Seidel1995upperbound}, which states that the number of faces of a $d$-dimensional polytope~$P$ with $n$ facets is asymptotically bounded by~$\mathcal{O}\left( n^{\lfloor \frac{d}{2} \rfloor} \right)$:
\begin{proposition}\label{cor:BoundOnVertices}
	In each iteration, the number of vertices of $D$ is in $\mathcal{O}\left(\lvert \Grid \rvert^{\lfloor \frac{d}{2}
 \rfloor}\right)$.
\end{proposition}
Second, we obtain a worst-case asymptotic running time of the vertex enumeration when using the (asymptotically optimal) vertex enumeration algorithm by \cite{Chazelle1993:convexhull} that runs, for a polyhedron of dimension~$d$ given as intersection of $n$~half-spaces, in $\mathcal{O}(n \log(n) + n^{\lfloor  \frac{d}{2} \rfloor})$. Since~$\dim(\Lambda) = d-1$ and, therefore, the  dimension of~$D$ is at most $d$, this implies:
\begin{corollary}\label{cor:RunningTimeVertexEnumeration}
	In each iteration, the running time of the vertex enumeration (Step~\ref{alg:FPTOAA:vertexenumeration} and \ref{alg:FPTOAA:vertex-enum2}) is in $$\mathcal{O}\left( \lvert\Grid \rvert \cdot \log( \lvert \Grid \rvert) + \lvert\Grid \rvert^{\lfloor \frac{d}{2} \rfloor} \right).$$
\end{corollary}
Note that the running time of both rounding schemes depends on the encoding length of the given weight vector~$\lambda^*$ as well, see Proposition~\ref{prop:Approx+Convex:GridRoundingCorrect+RunTime} and Proposition~\ref{prop:CorrectnessBoundaryRounding}. Hence, to prove polynomial running time of Algorithm~\ref{alg:FPTOAA}, we have to show that the encoding length of every extreme point obtained by the vertex enumeration algorithm is polynomially bounded in the instance size and the reciprocal of $\varepsilon$. Again an existing result, this time by~\cite{Groetschel2012geometric}, helps.
\begin{theorem}[\cite{Groetschel2012geometric}]\label{th:Approx+Convex:EncodingLengthOfLinearSystem}
	If a rational linear system $A x = b, x \in \R^{d}$ with $A \in \Q^{m \times d}$ and $b\in \Q^m$ has a unique solution $x^*$, then $x^* \in \Q^d$ and the encoding length of each component of~$x^*$ is polynomially bounded by the encoding length of~$A$ and~$b$.
\end{theorem}
This implies that the encoding length of each extreme point of~$D$ is bounded by a polynomial in the instance size and $\frac{1}{\varepsilon}$:
\begin{proposition}[\cite{Boekler:output-sensitive}]\label{prop:Approx+Convex:EncodingLengthExtremeVertices}
	Let~$S \subseteq X$. Then, the encoding length of each extreme point $(\lambda^*_1,\ldots, \lambda^*_d,z)$ of $D(S)$ is bounded polynomially in the instance size.
\end{proposition}
\begin{proof}
	Let $(\lambda^*_1,\ldots, \lambda^*_d,z)$ be an extreme point of $D(S)$. Then, there exists at most $d$ inequalities of the form $\lambda_i \geq  0$, $i = 1,\ldots d$, or $\sum_{i=1}^d \lambda_i f_i(x) - z \geq 0\ (\leq 0)$ in the case of minimization (maximization) induced by solutions $x^1, \ldots, x^{d} \in S$ such that $(\lambda^*_1,\ldots, \lambda^*_d,z)$ is the unique solution of the linear system
	\begin{align*}
	\begin{pmatrix}
	f_1(x^1), && \ldots&& f_d(x^1),&&-1\\
	f_1(x^2), && \ldots&& f_d(x^2),&&-1\\
	&& \vdots&& &&\\
	f_1(x^d), && \ldots&& f_d(x^d),&&-1\\
	\end{pmatrix}
	\cdot \begin{pmatrix}\lambda \\ z\end{pmatrix} \geq 0\ (\leq 0) \text{ and } \lambda &\in \Lambda.
	\end{align*}
	By assumption, the components of the images~$f(x^i)$, $i = 1,\ldots, d$, are polynomial-time computable in the instance size. Hence, the encoding length of the extreme point~$(\lambda^*_1,\ldots, \lambda^*_d,z^*)$ is polynomially bounded in the instance size by Theorem~\ref{th:Approx+Convex:EncodingLengthOfLinearSystem}. 
\end{proof} \noindent
We are now ready to prove termination and correctness of our algorithm.
\begin{theorem}\label{th:CorrectnessAndRunningTime}
	Algorithm~\ref{alg:FPTOAA} terminates and returns a $((1 + \varepsilon) \cdot \alpha)$-convex approximation set~$S$ in time
	\begin{align*}
	\mathcal{O} \left(T_{\LB/\UB}  +  \lvert \Grid \rvert ^{\lfloor \frac{d}{2} \rfloor +2}  + \lvert \Grid \rvert  \cdot T_{\ALG_{\alpha}} \right),
	\end{align*}
	where~$T_{\LB / \UB}$ denotes the time needed for computing the bounds~$\LB$ and~$\UB$, and~$T_{\ALG_{\alpha}}$ denotes the running time of~$\ALG_{\alpha}$.
\end{theorem}
\begin{proof}
	By Lemma~\ref{lem:Approx+Convex:Cardinality-grid} and Corollary~\ref{cor:BoundOnVertices}, the number of iterations of the while loop is in $\mathcal{O} \left( \lvert \Grid \rvert  \cdot \lvert \Grid \rvert ^{\lfloor \frac{d}{2} \rfloor} \right) =\mathcal{O} \left( \lvert \Grid \rvert ^{\lfloor \frac{d}{2} \rfloor + 1} \right)$, so the algorithm terminates.
	
	\medskip
	
	Next, we prove correctness. 
	Let $S$ be the set of solutions at termination and let $(\lambda^*_1, \ldots, \lambda^*_d, z^*)$ be an extreme point of $D(S)$. If $\lambda^* \coloneqq (\lambda_1,\ldots,\lambda_d) \notin P_<(I)$ for each $\emptyset \neq I \subseteq \{1, \dots, d\}$, no rounding has been applied and, thus, algorithm~$\ALG_{\alpha}$ is called for the weight vector~$\bar{\lambda} = \texttt{GridRounding}(\lambda^*,\varepsilon)$.
	Hence, there exists a solution $x^* \in \hat{X}$ such that $x^*$ is an $\alpha$-approximation for $\bar{\lambda}$. By Proposition~\ref{prop:Approx+Convex:GridRoundingCorrect+RunTime}, $x^*$ is a $((1 + \varepsilon) \cdot \alpha)$-approximation for $\lambda^*$. If $\lambda^* \in P_<(I)$ of some $\emptyset \neq I \subseteq \{1, \dots,d\}$, algorithm~$\ALG_{\alpha}$ is called for the weight vector $\bar{\lambda} = \texttt{GridRounding}(\lambda',\varepsilon')$, where $\lambda'$ is the weight vector returned by $\texttt{BoundaryRounding}(\lambda^*,\varepsilon',\LB,\UB)$. 
	Hence, there exists a solution $x^* \in S$ such that $x^*$ is an $\alpha$-approximation for $\bar{\lambda}$. Therefore, by Proposition~\ref{prop:Approx+Convex:GridRoundingCorrect+RunTime}, it holds that $x^*$ is an $\alpha$-approximation for~$\lambda'$, which, by Proposition~\ref{prop:CorrectnessBoundaryRounding}, implies that $x^*$ is a $(1 + \varepsilon') \cdot \alpha + \varepsilon'$-approximation for $\lambda$. Since~$(1 + \varepsilon') \cdot \alpha + \varepsilon' \leq (1 + \varepsilon) \cdot \alpha$, the solution~$x^*$ is a $((1 + \varepsilon \cdot \alpha)$-approximation for~$\lambda^*$. By Proposition~\ref{prop:SuffCondForApproxPoly}, we obtain that $S$ is a $((1 + \varepsilon) \cdot \alpha)$-convex approximation set.
	
	\medskip
	
	To conclude the proof, it is left to show the bound on the asymptotic worst-case running time. The vertex enumeration algorithm in Steps~\ref{alg:FPTOAA:vertexenumeration} and~\ref{alg:FPTOAA:vertex-enum2}, with an asymptotic running time of $\mathcal{O}\left( \lvert\Grid \rvert \cdot \log(\lvert\Grid\rvert) + \lvert\Grid\rvert^{\lfloor \frac{d}{2} \rfloor} \right)$ by Corollary~\ref{cor:RunningTimeVertexEnumeration}, is called whenever a new solution is found. Since $\ALG_{\alpha}$ is called at most $\mathcal{O}(\lvert\Grid\rvert)$ times, the accumulated running time of the vertex enumeration algorithm and the accumulated running time of $\ALG_{\alpha}$ are in
	$$\mathcal{O}\left( \lvert \Grid \rvert \cdot \left( \lvert\Grid \rvert \cdot \log( \lvert \Grid \rvert) + \lvert\Grid \rvert^{\lfloor \frac{d}{2} \rfloor} +  T_{\ALG_{\alpha}} \right) \right).$$
	Next, the remaining running time in each iteration is bounded. In each iteration, the $\texttt{BoundaryRounding}(\lambda^*,\varepsilon',\beta,\LB,\UB)$ algorithm is called and has constant running time by Proposition~\ref{prop:CorrectnessBoundaryRounding} and Proposition~\ref{prop:Approx+Convex:EncodingLengthExtremeVertices}. Note that \texttt{GridRounding} is only applied for weight vectors $\bar{\lambda} \notin P_<(I)$ for all $\emptyset \neq I \subseteq \{1, \dots, d\}$. 
	Consequently, by Proposition~\ref{prop:Approx+Convex:GridRoundingCorrect+RunTime} and Proposition~\ref{prop:Approx+Convex:EncodingLengthExtremeVertices}, the running time of $\texttt{GridRounding}(\lambda',\varepsilon',\LB,\UB)$ and $\texttt{GridRounding}(\lambda',\varepsilon,\LB,\UB)$ is in $\mathcal{O} \left(  \lvert \Grid \rvert \right)$.
	
	The set~$R$ can be represented as a matrix, where the entry $[a_1, \ldots, a_d]$ stores whether $\ALG_\alpha$ has been called for $(1, (1 + \varepsilon')^{a_1}, \ldots, (1 + \varepsilon')^{a_d})$. Hence, inserting an element to~$R$ and checking whether an weight vector~$\lambda$ is in~$R$ can be done in constant time.\footnote{Note that, in practice, it is preferable to implement~$R$ as a hash table.}
	Hence, the remaining running time of the while-loop is in $$\mathcal{O} \left( \lvert\Grid \rvert \cdot \lvert\Grid \rvert^{\lfloor \frac{d}{2} \rfloor} \cdot \lvert \Grid \rvert \right) =\mathcal{O} \left( \lvert\Grid\rvert^{\lfloor \frac{d}{2} \rfloor + 2} \right).$$
	In summary, the running time of Algorithm~\ref{alg:FPTOAA} is in
	\begin{align*}
	&\mathcal{O} \left(T_{\LB/\UB}  +  \lvert\Grid\rvert^{\lfloor \frac{d}{2} \rfloor +2}  + \lvert\Grid \rvert \cdot T_{\ALG_{\alpha}} \right).		
	\end{align*}
\end{proof}
\noindent
In particular, Theorem~\ref{th:CorrectnessAndRunningTime} yields:
\begin{corollary}
	Algorithm~\ref{alg:FPTOAA} yields an MFPTcAS if either an exact polynomial-time algorithm~$\ALG_{1}$ or an FPTAS is available for the weighted sum scalarization. If a PTAS is available for the weighted sum scalarization, Algorithm~\ref{alg:FPTOAA} yields an MPTcAS.
\end{corollary}

\section{Computational Experiments and Results}\label{sec:CompStudy}
In this section, we compare the performance of each available convex approximation algorithm on instances of the $3$-objective knapsack problem as well as the $3$-objective symmetric metric traveling salesman problem. For convenience, we denote by~$\DYAlg$, $\GridAlg$, and $\FPTOAAlg$ the algorithms of \cite{Diakonikolas:thesis}, \cite{Helfrich:mp-approximation}, and Algorithm~\ref{alg:FPTOAA}, respectively.

\medskip

For the knapsack problem, we follow~\cite{Bazgan+etal:fptas-mulicrit-knapsack} and consider two types of instances. The first type are \emph{uniform} instances for which the weight and the $i$th costs of each item are independently and uniformly sampled as integers in the interval~$[0,1000]$, and the capacity is set to half of the total weight of all items rounded up to the nearest integer. The second type of instances are \emph{conflicting} instances for which the weight of each item is independently and uniformly sampled as an integer in the interval~$[0,1000]$, the capacity is set to half of the total weight of all items rounded up to the nearest integer, and the costs of each item are sampled to be negatively correlated to each other. More precisely, for each item~$e$,~$f_1(e)$ is an integer uniformly generated in $[0,1000]$, $f_2(e)$ is an integer uniformly distributed in~$[1000 - f_1(e)]$, and $f_3(e)$ is an integer uniformly generated in~$[\max \{ 900 - f_1(e) - f_2(e), 0  \}, \min \{ 1100 - f_1(e) - f_2(e), 1000 -  f_1(e)\}]$. For each $n \in \{10,20,\ldots,250\}$ and each type, five $3$-objective knapsack instances are generated. For the weighted sum scalarization, the well-known \ExtGreedy~Algorithm (see, for example,~\cite{Kellerer+etal:book}) constitutes a $2$-approximation algorithm.

For the symmetric metric traveling salesman problem, we follow~\cite{Cornu:tsp,Florios:tsp,Lust:tsp,Paquete:tsp} and use \texttt{portgen} of the \text{DIMACS TSP} instance generator\footnote{\url{http://archive.dimacs.rutgers.edu/Challenges/TSP/}} to obtain, for each $i = 1,2,3$, integer coordinates of cities on a $1000 \times 1000$ square, on the basis of which the $i$th cost between each two cities is chosen to be the Euclidean distance of their $i$th coordinates.  For each $n \in \{10,20,\ldots,100\}$, five $3$-objective symmetric metric traveling salesman instances are generated. Note that each weighted sum scalarization of each $3$-objective symmetric metric traveling salesman instance is a single-objective symmetric metric traveling salesman instance, for which the well-known  \Chris~algorithm~\citep{Christofides:TSP}~constitutes a $\frac{3}{2}$-approximation algorithm.\footnote{Note that other types of traveling salesman instances are studied in \cite{Cornu:tsp,Florios:tsp,Lust:tsp,Paquete:tsp} as well. However, random instances or mixed instances do not necessarily yield metric instances and, thus, approximation algorithms with bounded approximation quality for the weighted sum scalarization of those instances do not exist unless $\textsf{N} = \textsf{NP}$, cf.~\cite{Williamson+Shmoys:Approximation}. This renders such instances unsuitable for our purposes. Additionally, clustered instances are not well-defined for $n < 100$. On the other hand, algorithms for obtaining optimal solution sets for the weighted sum scalarization, which can produce reference sets for determining various performance measures, do not finish within reasonable time on instances with $n \geq 100$ cities, which has also been observed in~\cite{Oezpeynirci2010ESN,Florios:tsp}.}

\medskip

To determine the a posteriori convex approximation quality of the returned solution sets~$S$, i.e., the smallest factor~$\beta$ for which $S$ is a $\beta$-convex approximation set, we introduce the appropriate variant of the $\varepsilon$-indicator~\citep{Zitzler2003}:
\begin{definition}
	 Let $\I = (X,f)$ of an instance of a $d$-objective minimization problem and let $S^*$ be an optimal solution set in $\I$ for the weighted sum scalarization. Then, the \emph{$\varepsilon$-convex indicator} of a set~$S \subseteq X$ of feasible solutions is defined by
	\begin{align*}
	\cInd(S) \coloneqq 
 \max_{\lambda \in \Lambda}\ \frac{\min_{x \in S} \ \lambda^\top f(x)}{ \min_{x^* \in S^*}\ \lambda^\top f(x^*)}.
	\end{align*}
	In an instance~$\I = (X,f)$ of a $d$-objective maximization problem with optimal solution set~$S^*$ for the weighted sum scalarization, the \emph{convex $\varepsilon$-indicator} of a set~$S \subseteq X$ of feasible solutions is defined by
	\begin{align*}
	\cInd(S) \coloneqq 
 \max_{\lambda \in \Lambda}\ \frac{ \max_{x^* \in S^*}\ \lambda^\top f(x^*)}{\max_{x \in S}\ \lambda^\top f(x)}.
	\end{align*}
\end{definition}
Note that our assumptions imply that, for each weight vector~$\lambda \in \Lambda$, there exists a solution~$x^\lambda$ that is  optimal for~$\lambda$. Hence, it does indeed suffice to compare solution sets~$S$ with an (arbitrary) optimal solution set~$S^*$ for the weighted sum scalarization.

\medskip

All algorithms have been implemented using \texttt{Python 3.7 (Conda 4.8)}, where vertex enumerations are performed with the algorithm \texttt{QHull}~\citep{Barber:QuickhullAlgorithm} provided by the package \texttt{scipy~1.4.1}. An implementation of \Chris~is provided by the package \texttt{networkx~2.6.3}. 
In order to find an optimal solution set~$S^*$ for the weighted sum scalarization for each instance, we have also implemented exact algorithms for the weighted sum scalarizations based on \texttt{Gurobi~9.1.1} and the dual variant of Benson's Outer Approximation Algorithm, and implemented a filtering method to remove redundant solutions. All instances and the source code are available via \git.  

\medskip 
 
All experiments have been performed on a computer server equipped with two Intel(R) Xeon(R) CPUs E5-2670 (single processor specifications: nominal speed~2.9GHz, boost up to 3.8GHz, 8~cores, 16~threads, 20MB~Cache) and 192GB DDR-3 ECC RAM at 1333MHz using the operating system environment Ubuntu Linux~11. In each run, a time limit of 1~hour has been set.  To compare the performances of the convex approximation algorithms, we report 
\begin{itemize}
    \item the average running time in seconds (s) of 10 runs,
	\item the $\varepsilon$-convex indicator of the returned set~$S$,
	\item the ratio of the cardinality of the returned set~$S$ to the cardinality of the  solution set~$S^*$ for the weighted sum scalarization obtained by the dual variant of Benson's Outer Approximation Algorithm with the filtering method,
\end{itemize}
for each combination of the convex approximation algorithms and~$\varepsilon \in \{0.1,0.25,0.5\}$.

\medskip

The results are summarized in Figure~\ref{fig:results_knapsack} and Figure~\ref{fig:results_tsp} for the knapsack and traveling salesman problem, respectively. Here, we note that, despite its importance in theory, no run of \DYAlg~finished within 1~hour on any instance and any $\varepsilon$, so \DYAlg~is omitted in the following discussion.

\begin{figure}
    \centering
    \includegraphics[width = \textwidth]{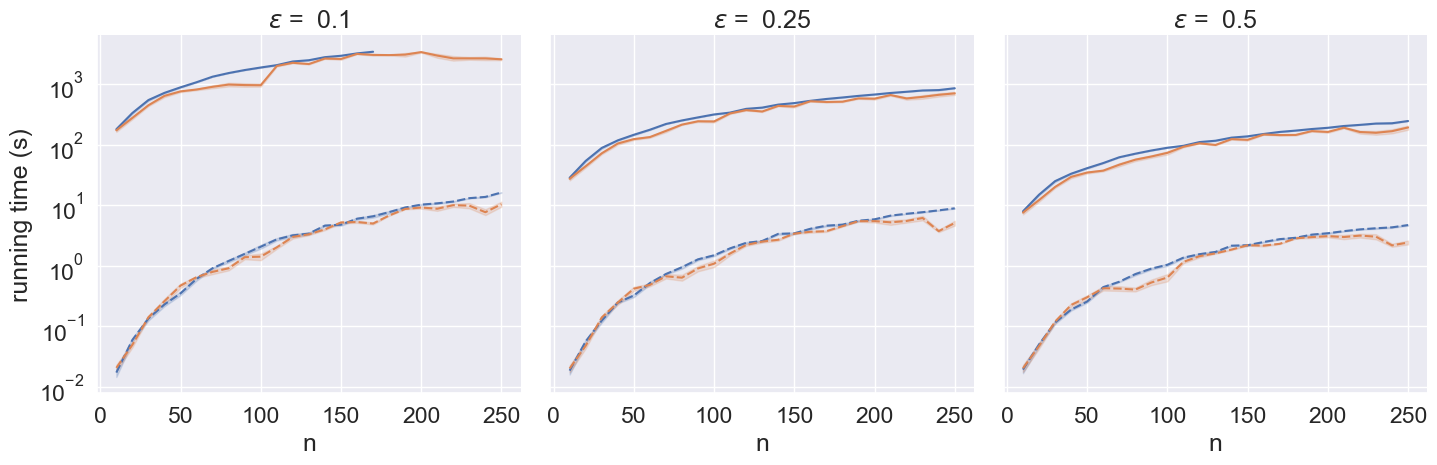}
    \includegraphics[width = \textwidth]{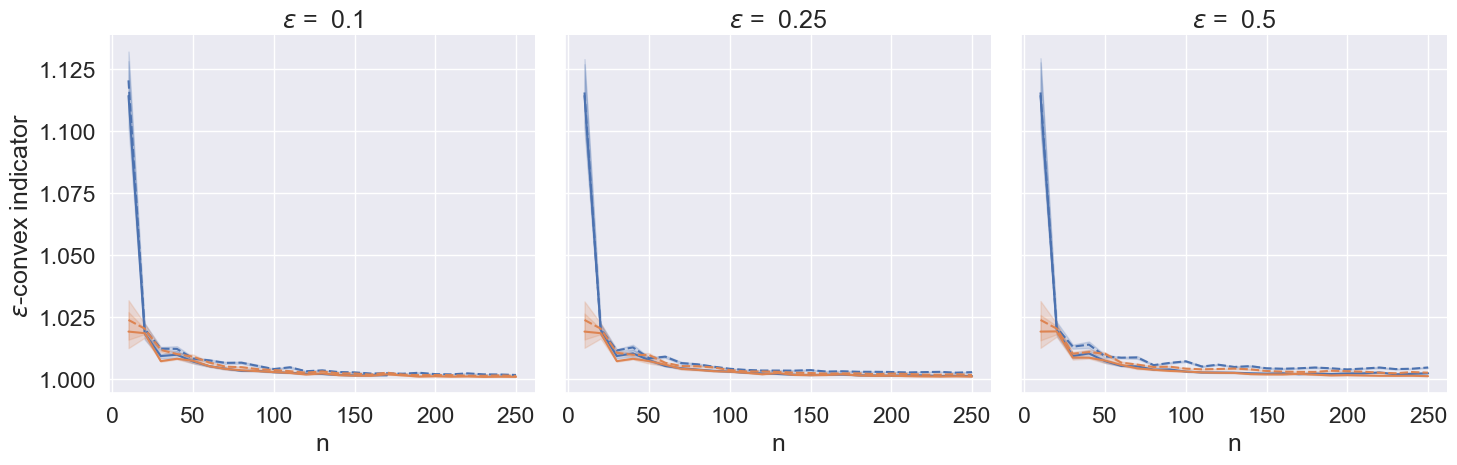}
    \includegraphics[width = \textwidth]{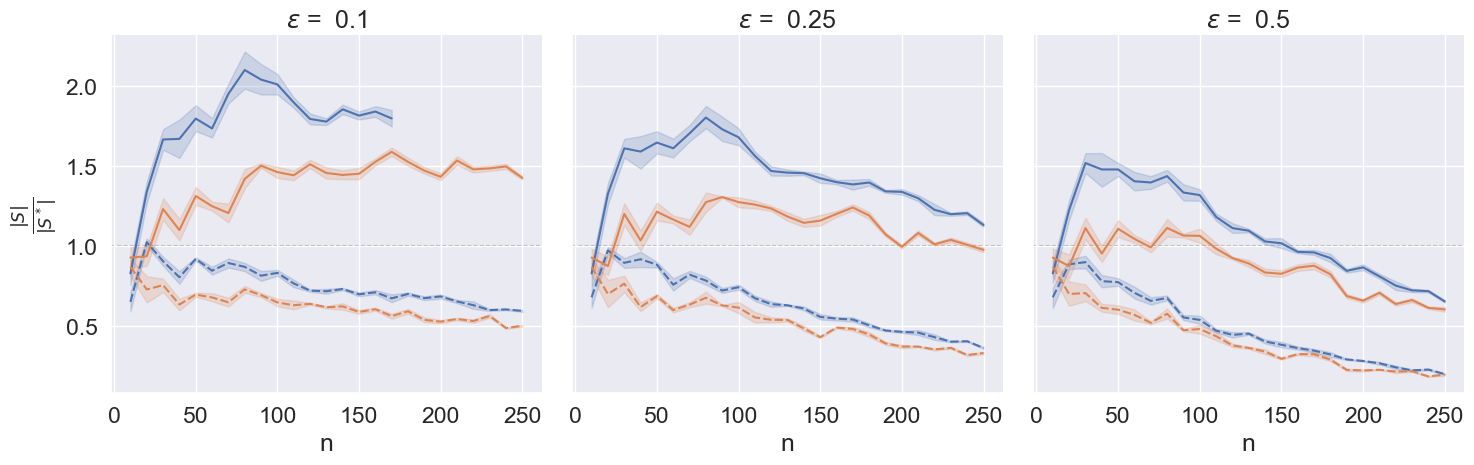}
    \caption{Results for multiobjective knapsack instances. solid line: \GridAlg, dashed line: \FPTOAAlg, orange: uniform instances, blue: conflicting instance.}
    \label{fig:results_knapsack}
\end{figure}
\begin{figure}
    \centering
    \includegraphics[width = \textwidth]{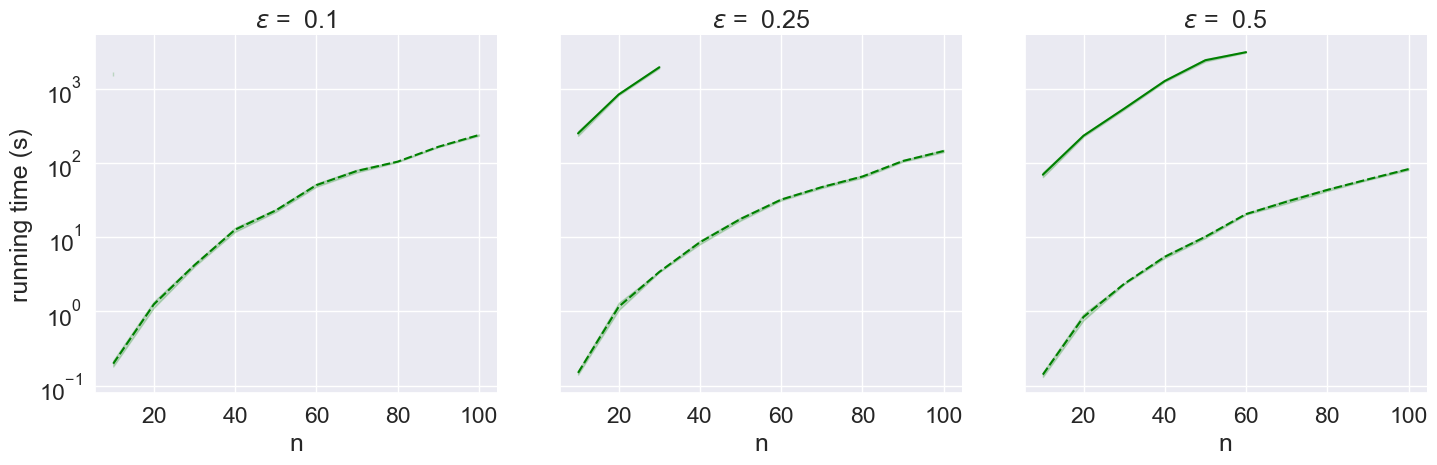}
    \includegraphics[width = \textwidth]{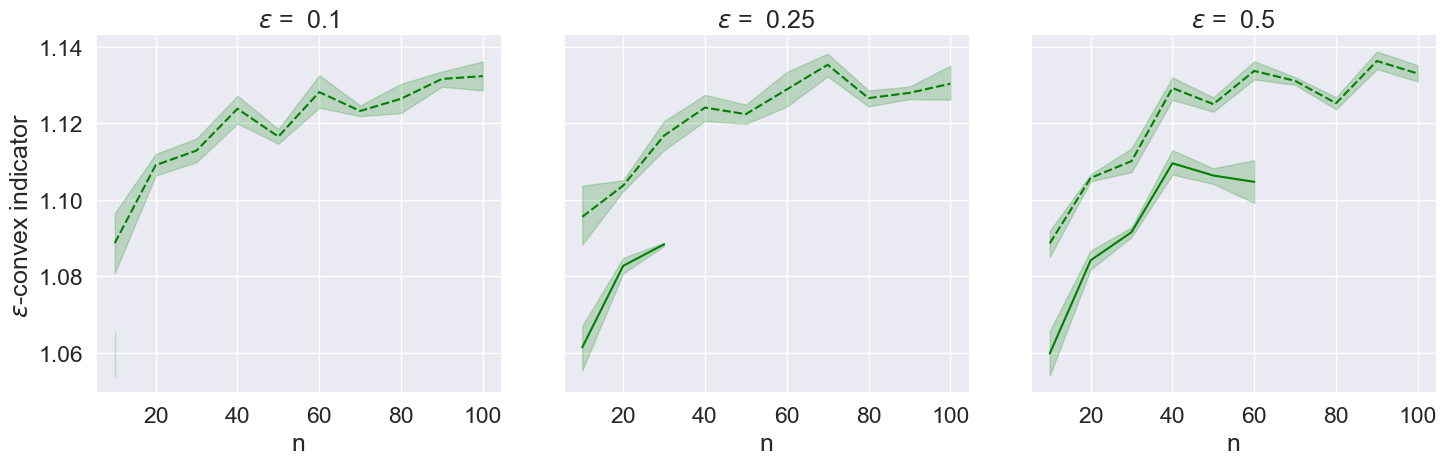}
    \includegraphics[width = \textwidth]{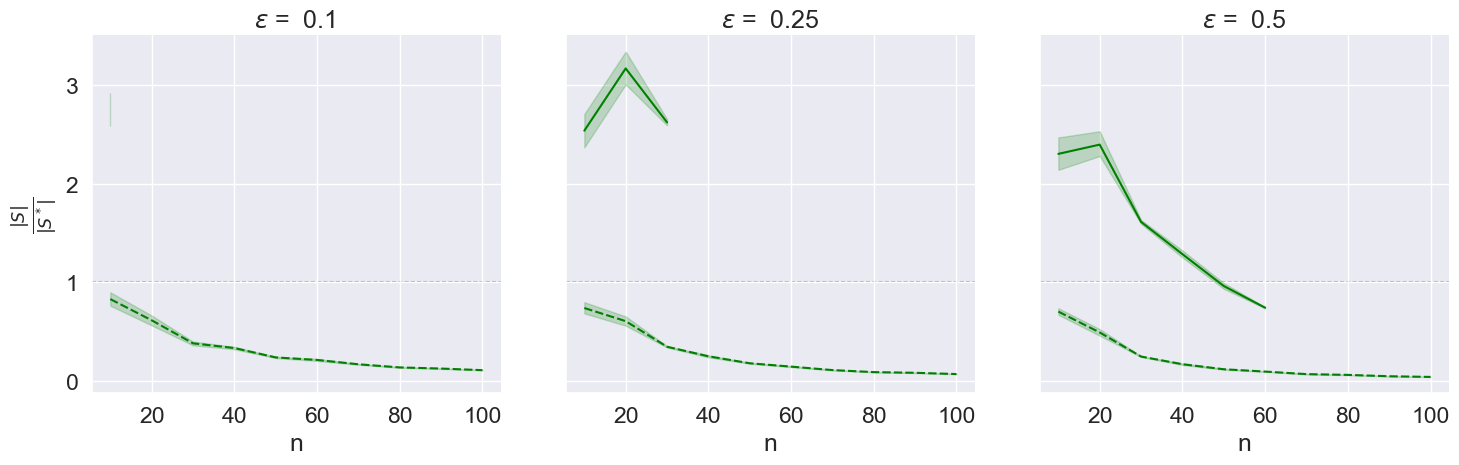}
    \caption{Results for multiobjective symmetric metric traveling salesman instances. solid line: \GridAlg, dashed line: \FPTOAAlg.}
    \label{fig:results_tsp}
\end{figure}

With respect to average running time, we can observe that \GridAlg~is significantly slower than \FPTOAAlg~on each instance of both problem types. This is also settled by the fact that \FPTOAAlg~terminated within 1 hour in each run, whereas several runs of \GridAlg~did not: knapsack instances with $n \geq 160$ and $\varepsilon = 0.1$, as well as traveling salesman instances with $n \geq 30$ and $\varepsilon = 0.1$, $n \geq 50$ and $\varepsilon=0.25$, $n \geq 80$ and $\varepsilon=0.5$ could not be solved within the time limit. Lastly, for each problem type and convex approximation algorithm, we can observe that the average running time increases with increasing~$n$, and is inverse proportional to~$\varepsilon$.

In each instance of each problem type and for every combination of~$\varepsilon$ and convex approximation algorithm, the $\varepsilon$-convex indicator is  significantly smaller than the theoretical bounds: for knapsack instances, the $\varepsilon$-convex indicator is always less that $1.13$, whereas the theoretical upper bounds satisfy~$(1 + \varepsilon) \cdot 2 \geq 2.2$ for every~$\varepsilon \in \{0.1,0.25,0.5\}$. For traveling salesman instances, the $\varepsilon$-convex indicator is always less that $1.2$, contrasting the theoretical upper bounds~$(1 + \varepsilon) \cdot 1.5 \geq 1.65$ for every~$\varepsilon \in \{0.1,0.25,0.5\}$. Moreover, observe that \GridAlg~returns, if terminated within the time limit, solution sets achieving a slightly better $\varepsilon$-convex indicator in comparison to solution set returned by $\FPTOAAlg$. Lastly, for each problem type and convex approximation algorithm, we can observe that the $\varepsilon$-convex indicator is not significantly affected by variations off~$n$ and increases slightly with increasing~$\varepsilon$.

Concerning the cardinality of the returned solution sets, we can observe that,
for each problem type and each $\varepsilon$, the cardinality of the convex approximation
set returned by \FPTOAAlg~is always smaller than the cardinality of the one returned
by \GridAlg. In addition, the cardinality of the convex approximation set returned
by \FPTOAAlg~is smaller than the cardinality of the reference
optimal solution set for the weighted sum scalarization in each instance. For \GridAlg, this can only
be observed for larger size instances. Nevertheless, we can observe that, for each problem type and convex approximation algorithm, the ratio of the cardinality of the returned sets to the cardinality of the optimal solution sets for the weighted sum scalarization is inverse proportional to~$n$ and~$\varepsilon$.

In summary, these results indicate that \FPTOAAlg~significantly outperforms
every other existing convex approximation algorithm in terms of running time
and cardinality of the returned convex approximation set while still providing solution sets achieving comparable $\varepsilon$-convex indicators. Additionally, the $\varepsilon$-convex indicators are noticeably smaller than the theoretical bounds. 

\section{Conclusion}
Multiobjective minimization problems can be approximated well by means of optimal solution sets for the weighted sum scalarization. However, for multiobjective maximization problems, strong impossibility results are 
known in this context. This polarity does not exist anymore when convex approximation sets are considered, though the efficient computation of such convex approximation sets has mainly been studied from a theoretical point of view so far.

This paper presents an algorithm to compute convex approximation sets that (1)~relies on an efficient exact or approximate algorithm for the weighted sum scalarization and is, thus, applicable to a large variety of multiobjective optimization problems, (2)~returns polynomial-sized solution sets that constitute convex approximation sets with approximation quality arbitrarily close to the approximation quality of the weighted sum algorithm, (3)~yields an M(F)PTcAS if a polynomial-time exact algorithm or an (F)PTAS is used for the weighted sum scalarization, and (4)~outperforms all other existing convex approximation algorithms in terms of practical running time and cardinality of the returned solution sets. Property~(4) is demonstrated in the first comprehensive computational study of convex approximation algorithms conducted so far.
Consequently, this paper initiates a benchmarking for general convex approximation algorithms and lays the foundation for such algorithms to gain practical importance alongside heuristics and exact methods.

Convex approximation sets as studied so far aim to achieve the same approximation quality in all objectives. However, polynomially-sized (classic) approximation sets that are exact in one objective exist under mild assumptions as well and can be computed efficiently if (and only if) the so-called \emph{dual restrict problem} can be solved efficiently~\citep{Herzel+etal:dualrestrict}. Thus, convex approximation sets that are exact in one (or several) objectives exist and can be computed efficiently via the dual restrict problem under the same assumptions. However, it is unclear whether the polynomial-time solvability of the dual restrict problem is also a necessary condition for the polynomial-time computation of convex approximation sets.
Other possible directions for future research concern further improvements of our algorithm. Here, possible approaches are double description methods, or designing versions of our approach that are specially tailored to classes of problems with a specific structure.

\section*{Acknowledgements}
This work was funded by the Deutsche Forschungsgemeinschaft (DFG, German Research Foundation) – Project number~398572517. We also acknowledge the valuable assistance of Nico Gerber and Carla Eva Hamm in carrying out the computational study presented in Section~\ref{sec:CompStudy}.

\bibliography{main}

\begin{appendices}
\section{Invariant of Sorting w.r.t.\ Lifting}\label{lem:P_<(I)-rounding-invariant}
	Let~$\lambda \in \R^d_\geq$ such that $\lambda_1 \leq \lambda_2 \leq \ldots, \lambda_d$. Further, let~$I = \{1,\ldots,k\}$ be an index set for some~$k \in \{ 1,\ldots d-1\}$ such that $\lambda \in P_<(I)$. Let~$\bar{\lambda} \in \R^d_\geq$ be defined as in Equation~\eqref{eq:BoundaryRounding} with index set~$I$. Then, $\bar{\lambda}_1 \leq \bar{\lambda}_2 \leq \ldots, \bar{\lambda}_d$ as well. Further,
	\begin{itemize}
		\item if $\lambda \notin P_<(\{1,\ldots,k'\})$ for some~$1 \leq k' < k$, then $\bar{\lambda} \notin P_<(\{1,\ldots,k'\})$,
		\item if $\lambda \in P_=(\{1,\ldots,k'\})$ for some~$1 \leq k' < k$ and $\sum_{j=1}^k \lambda_j > 0$, then $\bar{\lambda} \in P_=(\{1,\ldots,k'\})$.
	\end{itemize}
\begin{proof}
We first  prove that $\bar{\lambda}_1 \leq \bar{\lambda}_2 \leq \ldots \leq \bar{\lambda}_d$. 
If $\sum_{j=1}^k \bar{\lambda}^{k}_j >0$, it is, for $i = 1, \ldots, k - 1$,
\begin{align*}
\bar{\lambda}_i =  \frac{\lambda_i}{\sum_{j=1}^k \lambda_j} \cdot c \cdot \lambda_{k + 1} \leq \frac{\lambda_{i+1}}{\sum_{j=1}^k \lambda_j} \cdot c \cdot \lambda_{k + 1} = \bar{\lambda}_{i+1}
\end{align*}
and
\begin{align*}
\bar{\lambda}_k =  \frac{\lambda_k}{\sum_{j=1}^k \lambda_j} \cdot c \cdot \lambda_{k + 1} \leq \lambda_{k +1} = \bar{\lambda}_{k + 1}.
\end{align*}
If $\lambda_j=0$ for all $j=1, \ldots, k$, it holds, for $i = 1, \ldots, k - 1$, that
\begin{align*}
\bar{\lambda}_i = \frac{1}{k +1} \cdot c \cdot \lambda_{k +1} = \bar{\lambda} _{i + 1}
\end{align*}
and, since $c \in (0,1)$,
\begin{align*}
\bar{\lambda}_k = \frac{1}{k +1} \cdot c \cdot \lambda_{k +1} \leq \lambda_{k +1} = \bar{\lambda}_{k + 1}.
\end{align*}
In both cases it holds, for $i = k + 1, \ldots, d-1$, that $\bar{\lambda}_i = \lambda_i \leq \lambda_{i+1} = \bar{\lambda}_{i+1}$.

\medskip

If, additionally, $\lambda \notin P_<(\{1,\ldots,k'\})$ for some~$1 \leq k' < k$, it follows that $\sum_{i = 0}^{k'} \lambda_i \geq c \cdot \lambda_{k' +1}$. Since $k' +1 \leq k$, this implies that
\begin{align}\label{eq:lem:P_<(I)-rounding-invariant}
\sum_{i=1}^{k'} \bar{\lambda}_i = \sum_{i=1}^{k'} \frac{\lambda_i}{\sum_{j=1}^k \lambda_j} \cdot c \cdot \lambda_{k + 1} \geq \frac{ c \cdot \lambda_{k'+1}}{\sum_{j=1}^k \lambda_j} \cdot c \cdot \lambda_{k + 1} = c \cdot  \bar{\lambda}_{k'+1}
\end{align}
and
\begin{align*}
\sum_{i=1}^{k'} \bar{\lambda}_i = \sum_{i=1}^{k'} \frac{1}{k +1} \cdot c \cdot \lambda_{k +1} \geq \frac{1}{k +1} \cdot c \cdot \lambda_{k +1} = \bar{\lambda}_{k+1}
\end{align*}
in the case of $\sum_{j=1}^k \lambda_j > 0$ and $\lambda_j=0$ for all $j=1,\ldots, k$, respectively. 
The second statement holds true since Equation~\eqref{eq:lem:P_<(I)-rounding-invariant} holds then with equality.	
\end{proof}

\section{Proof of Proposition~\ref{prop:CorrectnessBoundaryRounding}}\label{proof-prop-correctness-boundary}
\begin{proof}[]
~\\
    Termination is straightforward. The permutation~$\sigma$ and its inverse~$\sigma^{-1}$ can be determined in $\mathcal{O}(d \cdot \log(d))$ time. The for-loop of Steps~\ref{alg:BoundaryRounding:BeginForLoop}-\ref{alg:BoundaryRounding:EndForLoop} can be realized in $\mathcal{O}(d^2)$ time. All remaining operations run in~$\mathcal{O}(d)$. Since $d$ is fix, asymptotic constant worst-case running time is proven.

    \medskip

    Next, we prove correctness.	If $\flag = \text{FALSE}$, no rounding has been applied, $\lambda' = \lambda$ and every $\beta$-approximation for $\lambda'$ is trivially  a $\beta$-approximation for $\lambda$. Hence, assume that $\flag = \text{TRUE}$ holds true in the following. That is, the if-condition in Step~\ref{alg:BoundaryRounding:IfConditionPI} is satisfied for at least one $k \in \{1, \ldots, d-1\}$.
	
	For $k = 1, \ldots, d-1$, denote by $\bar{\lambda}^k$ the state of the weight vector~$\bar{\lambda}$ at the beginning of iteration~$k$ in the for-loop of Steps~\ref{alg:BoundaryRounding:BeginForLoop}-\ref{alg:BoundaryRounding:EndForLoop}. Further, denote by $\bar{\lambda}^d$ the state of the weight vector~$\bar{\lambda}$ after iteration $d-1$. We prove that the following statements holds true for $k = 2,\ldots,d$:
	\begin{enumerate}
		\item $\bar{\lambda}^k_1 \leq \bar{\lambda}^k_2 \leq \ldots \leq \bar{\lambda}^k_d$.\label{prop:CorrectnessBoundaryRounding:proof:Sorting}
		
		\item $\bar{\lambda}^k \notin P_<(\{0, \ldots, k'\}$ for every $1 \leq k' \leq k-1$.\label{prop:CorrectnessBoundaryRounding:proof:NotInP<}
		
		\item $\bar{\lambda}^k_i \geq \bar{\lambda}^{k-1}_i$ for $i = 1, \ldots, k-1$ and $\bar{\lambda}^k_j = \bar{\lambda}^{k-1}_j$ for $j = k+1,\ldots,d$.\label{prop:CorrectnessBoundaryRounding:proof:greaterequal}
		
		\item Let $R^{k} \subseteq \{0,\dots, k-1\}$ be the set of iterations up to $k-1$ such that the if-condition in Step~\ref{alg:BoundaryRounding:IfConditionPI} has been true. Then, $$\bar{\lambda}^k \in P_=(\{0,\ldots,k'\}) \text{ for all } k' \in R^k$$
		and, in particular, 
		$$\bar{\lambda} \in \conv\left( \{\bar{\lambda}^k\} \cup \{ \proj^{\{0,\dots,k'\}}(\bar{\lambda}^k), k' \in R^k\}\right).$$\label{prop:CorrectnessBoundaryRounding:proof:Conv}
	\end{enumerate}
	Statement~\ref{prop:CorrectnessBoundaryRounding:proof:Sorting} together with Observation~\ref{obs:sorting-P_<(I)} guarantees that, if $\bar{\lambda}^k \in P_<(I)$ for some $k$ and some index set $I \subseteq \{1,\ldots, d\}$, it must be that $\bar{\lambda}^k \in P_<(\{0,\ldots,\lvert I \rvert \})$, and that Step~\ref{alg:BoundaryRounding:InverseSorting} yields indeed the correctly rounded weight vector. 
	Statement~\ref{prop:CorrectnessBoundaryRounding:proof:NotInP<} guarantees that $\bar{\lambda}^d$ satisfies $\bar{\lambda}^d \notin P_<(I)$ for every $\emptyset \neq I \subseteq \{1,\ldots, d\}$ (note that~$P_<(\{1,\ldots, d\})$ does not contain any weight vector). This fact is invariant with respect to rearranging the components and multiplication of the weight vector with a positive scalar. Hence, both $\tilde{\lambda}$ and $\lambda'$ satisfy~$\tilde{\lambda}, \lambda' \notin P_<(I)$ for every $\emptyset \neq I \subseteq \{1,\ldots, d\}$ as well. This yields~$\lambda' \in \Compact$.
	Statement~\ref{prop:CorrectnessBoundaryRounding:proof:greaterequal} is an auxiliary statement to prove Statement~\ref{prop:CorrectnessBoundaryRounding:proof:Conv}. It states that a component is never decreased during the for-loop.
	Statement~\ref{prop:CorrectnessBoundaryRounding:proof:Conv} together with Lemma~\ref{lem:ApproximationGuaranteeWithinTreshhold} implies that every $\beta$-approximation for $\bar{\lambda}^d$ is a $(\beta + \varepsilon')$-approximation for each weight vector in~$\proj^{\{1,\ldots, k'\}}(\bar{\lambda}^d)$, $k' \in R^{d}$. Thus, by Lemma~\ref{lem:Approx+Convex:convexity}, every $\beta$-approximation for $\bar{\lambda}^d$ is a $(\beta + \varepsilon')$-approximation for $\bar{\lambda}$. Again, rearranging the components of both~$\bar{\lambda}$ and~$\bar{\lambda}^d$ according to~$\sigma^{-1}$ does not change this, so every $\beta$-approximation for $\lambda'$ is a $(\beta + \varepsilon')$-approximation for $\lambda$.
	Hence, to prove correctness of Algorithm~\ref{alg:Approx+Convex:GridRounding}, it is left to prove Statements~\ref{prop:CorrectnessBoundaryRounding:proof:Sorting}-\ref{prop:CorrectnessBoundaryRounding:proof:Conv}.	
	
	\medskip
	
	\noindent
	Statements~\ref{prop:CorrectnessBoundaryRounding:proof:Sorting} and \ref{prop:CorrectnessBoundaryRounding:proof:NotInP<} follow by Lemma~\ref{lem:P_<(I)-rounding-invariant}. We prove Statements~\ref{prop:CorrectnessBoundaryRounding:proof:greaterequal} and \ref{prop:CorrectnessBoundaryRounding:proof:Conv} by induction over $k$.
	
	Let $k = 2$. If $\bar{\lambda}^1_1 \geq c \cdot \bar{\lambda}^1_{2}$, it holds that $\bar{\lambda}^{2} = \bar{\lambda}^1$, so there is nothing to prove. Let $\bar{\lambda}^1_1 < c \cdot \bar{\lambda}^1_{2}$.	
	In the case of $ \bar{\lambda}^{1}_1 >0$, it follows  that  $\bar{\lambda}^{2}_1 = \frac{\bar{\lambda}^{1}_1}{\bar{\lambda}^{1}_1} \cdot c \cdot \bar{\lambda}^{1}_{2} > \bar{\lambda}^{1}_1$. In the case of $\bar{\lambda}_1=0$, it holds trivially that $\bar{\lambda}^{2}_1 \geq 0$. Since $\bar{\lambda}^{2}_i = \bar{\lambda}^{1}_i$ for $i = 2,\ldots, d$, Statement~\ref{prop:CorrectnessBoundaryRounding:proof:greaterequal} follows.
	Statement~\ref{prop:CorrectnessBoundaryRounding:proof:Conv} follows by Lemma~\ref{lem:boundaryRounding}.
	
	Assume, that Statements~\ref{prop:CorrectnessBoundaryRounding:proof:Sorting}-\ref{prop:CorrectnessBoundaryRounding:proof:Conv} hold for some $1 \leq k \leq d-1$. If $\sum_{i=1}^k \bar{\lambda}^k_i \geq c \cdot \bar{\lambda}^k_{k +1}$, it holds that $\bar{\lambda}^{k +1} = \bar{\lambda}^k$ and Statements~\ref{prop:CorrectnessBoundaryRounding:proof:Sorting}-\ref{prop:CorrectnessBoundaryRounding:proof:Conv} hold by induction hypothesis. Hence, it can be assumed that $\sum_{i=1}^k \bar{\lambda}^k_i < c \cdot \bar{\lambda}^k_{k +1}$ and, therefore,
	\begin{align*}
	\bar{\lambda}^{k + 1} = 
	\begin{cases}
	\frac{\bar{\lambda}^{k}_i}{\sum_{j=1}^k \bar{\lambda}^{k}_j} \cdot c \cdot \bar{\lambda}^{k}_{k + 1}, &\text{ if } i = 1,\ldots, k \text{ and } \sum_{j=1}^k \bar{\lambda}^{k}_j >0,\\
	\frac{1}{k + 1} \cdot c \cdot \bar{\lambda}^{k}_{k + 1}, &\text{ if } i = 1, \ldots, k \text{ and } \bar{\lambda}^{k}_j =0 \text{ for all } j=1 \ldots, k,\\
	\bar{\lambda}^k_i, &\text{ if } i \geq k +1.
	\end{cases}
	\end{align*} holds true in the following.
	\begin{enumerate} 


		\item[3.] Let $i \in \{1, \ldots, k\}$. Since $\sum_{i=1}^k \bar{\lambda}^k_i < c \cdot \bar{\lambda}^k_{k +1}$, it follows in the case of $\sum_{j=1}^k \bar{\lambda}^{k}_j >0$ that $\bar{\lambda}^{k +1}_i = \frac{\bar{\lambda}^{k}_k}{\sum_{j=1}^k \bar{\lambda}^{k}_j} \cdot c \cdot \bar{\lambda}^{k}_{k + 1} \geq \bar{\lambda}^{k}_i$. In the case of $\bar{\lambda}^k_j=0$ for all $j=1, \ldots, k$, it holds trivially that $\bar{\lambda}^{k + 1}_i \geq 0$. Since $\bar{\lambda}^{k + 1}_i = \bar{\lambda}^{k}_i$ for $i = k +1,\ldots, d$, Statement~\ref{prop:CorrectnessBoundaryRounding:proof:greaterequal} follows.
		
		\item[4.] If $\bar{\lambda}^k_j=0$ for all $j=1, \ldots, k$, it must be that $\bar{\lambda} \notin P_\leq(\{1,\ldots,k'\})$ for every $1 \leq k' \leq k - 1$. Since $\sum_{j=1}^k \bar{\lambda}^{k}_j < c \cdot \bar{\lambda}_{k +1}$ holds true by assumption, it follows that $R^{k +1} = \{k\}$. Hence, it is to show that
		\begin{align*}
		\bar{\lambda}^{k + 1} \in P_= (\{1, \ldots, k\}) \text{ and } \bar{\lambda} \in \conv\{ \bar{\lambda}^{k+1}, \proj^{\{1,\ldots, k\}}(\bar{\lambda}^{k+1}) \}.
		\end{align*}
		It holds that
		\begin{align*}
		\sum_{i=1}^{k} \bar{\lambda}^{k + 1}_i = \sum_{i=1}^{k} \frac{1}{k +1} \cdot c \cdot \bar{\lambda}^k_{k +1} = c \cdot \bar{\lambda}^k_{k +1} = c \cdot \bar{\lambda}^{k + 1}_{k +1},
		\end{align*}
		and, therefore, $\bar{\lambda}^{k + 1} \in P_=(\{1,\dots, k\})$. Furthermore, since $\bar{\lambda} \notin P_\leq(\{1,\ldots,k'\})$ for each $1 \leq k' \leq k - 1$, the if-condition in Step~\ref{alg:BoundaryRounding:IfConditionPI} has not been been true for each $k' < k$ and, thus, it must be that $\bar{\lambda} = \bar{\lambda}^k$. Then, Statement~\ref{prop:CorrectnessBoundaryRounding:proof:greaterequal} states that $\bar{\lambda}^{k +1}_i \geq \bar{\lambda}_i$ for all $i = 1,\ldots, d$, which implies that $\bar{\lambda} \in \conv\{ \bar{\lambda}^{k+1}, \proj^{\{1,\ldots, k\}}(\bar{\lambda}^{k+1}) \}$.\\
		
		Consider now the case $\sum_{j=1}^k \bar{\lambda}^k_j > 0$. Note that $R^{k +1} = R^{k} \cup \{ k \}$. Let $\bar{\lambda}^{k} \notin P_=(\{1,\ldots,k'\})$ for some $k' \leq k -1$. Then, $\sum_{j=1}^{k'} \bar{\lambda}^{k}_i = c \cdot \bar{\lambda}^k_{k'+1}$. Since, $k' +1 \leq k$, it follows that
		\begin{align*}
		\sum_{i=1}^{k'} \bar{\lambda}^{k + 1}_i = \sum_{i =1}^{k'} \frac{\bar{\lambda}^{k}_i}{\sum_{j = 1}^k \bar{\lambda}^{k}_j} \cdot c \cdot \bar{\lambda}^{k}_{k + 1} = \frac{c \cdot \bar{\lambda}^{k}_{k'+1}}{\sum_{j = 1}^k \bar{\lambda}^{k}_j} \cdot c \cdot \bar{\lambda}^{k}_{k + 1} = c \cdot \bar{\lambda}^{k +1}_{k'+1}
		\end{align*}
		and, therefore, $\bar{\lambda}^{k +1} \in P_=(\{1,\ldots,k'\})$ for all $k' \in R^k$. Furthermore,
		\begin{align*}
		\sum_{i=1}^k \bar{\lambda}^{k + 1}_i = \sum_{i = 1}^k \frac{\bar{\lambda}^{k}_i}{\sum_{j = 1}^k \bar{\lambda}^{k}_j} \cdot c \cdot \bar{\lambda}^{k}_{k + 1} = c \cdot \bar{\lambda}^{k}_{k +1} = \bar{\lambda}^{k +1}_{k + 1}.
		\end{align*}
		This yields $\bar{\lambda}^k \in P_=(\{1,\ldots,k'\}) \text{ for all } k' \in R^{k +1 }$ and it is left to show that $$\bar{\lambda} \in \conv\left( \{  \bar{\lambda}^{k +1}\} \cup \{ \proj^{\{1,\dots,k'\}}(\bar{\lambda}^k), k' \in R^{k + 1}\}\right).$$
		By induction hypothesis, we know that $$\bar{\lambda} \in \conv\left( \{\bar{\lambda}^k\} \cup \{ \proj^{\{1,\ldots,k'\}}(\bar{\lambda}^k), k' \in R^k\} \right),$$ which means that there exist coefficients $\theta \in [0,1]$ and $\theta_{k'} \in [0,1], k' \in R^k$ such that
		\begin{align*}
		\bar{\lambda} = \theta \cdot \bar{\lambda}^k + \sum_{k' \in R^k} \theta_{k'} \cdot \proj^{\{1,\ldots, k'\}}(\bar{\lambda}^k) \text{ and } \theta + \sum_{k' \in R^k} \theta_{k'} = 1.
		\end{align*}
		Statement~\ref{prop:CorrectnessBoundaryRounding:proof:greaterequal} implies that $\bar{\lambda}^k \in \conv\left( \{\bar{\lambda}^{k +1},\proj^{\{1, \ldots,k\}}(\bar{\lambda}^{k +1})\}\right)$, i.e., there exists some $\mu \in [0,1]$ such that
		\begin{align*}
		\bar{\lambda}^k = \mu \cdot \bar{\lambda}^{k +1} + (1 - \mu) \cdot \proj^{\{1,\ldots, k\}}(\bar{\lambda}^{k +1}).
		\end{align*}
		Note that, since $\{1,\ldots, k'\} \subseteq \{1,\ldots, k\}$ for $k' \in R^k$, it holds that
		\begin{align*}
		\proj^{\{1,\ldots,k'\}} \left( \proj^{\{1,\ldots, k\}}(\bar{\lambda}^{k +1})\right) = \proj^{\{1,\ldots, k\}}(\bar{\lambda}^{k +1})
		\end{align*}
		and, thus,
		\begin{align*}
		\bar{\lambda} &= \theta \cdot \bar{\lambda}^{k} +  \sum_{k' \in R^k} \theta_{k'} \cdot \proj^{\{1,\ldots,k'\}} (\bar{\lambda}^{k})\\
		&= \theta \cdot \left( \mu \cdot \bar{\lambda}^{k +1} + (1 - \mu) \cdot \proj^{\{1,\ldots,k\}}(\bar{\lambda}^{k +1})  \right) \\
		&\hspace{1.2cm}+ \sum_{k' \in R^k} \theta_{k'} \cdot \proj^{\{1,\ldots,k'\}} \left(  \mu \cdot \bar{\lambda}^{k+1} + (1 - \mu) \cdot \proj^{\{1,\ldots,k\}}(\bar{\lambda}^{k+1})  \right)\\
		&= \theta \cdot \mu  \cdot  \bar{\lambda}^{k+1} + \theta \cdot (1 - \mu) \cdot \proj^{\{1,\ldots,k\}}(\bar{\lambda}^{k +1}) \\
		&\hspace{2.33cm}+\sum_{k' \in R^k} \theta_{k'} \cdot \mu \cdot \proj^{\{1,\ldots,k'\}}( \bar{\lambda}^{k +1} ) \\
		&\hspace{2.33cm}  +\sum_{k' \in R^k} \theta_{k'} \cdot (1 - \mu) \cdot  \proj^{\{1,\ldots,k'\}}\left(\proj^{\{1,\ldots,k\}}(\bar{\lambda}^{k +1})\right)\\
		&=\theta \cdot \mu  \cdot  \bar{\lambda}^{k+1}  + \sum_{k \in R^k} \theta_{k'} \cdot \mu \cdot \proj^{\{1,\ldots,k'\}}( \bar{\lambda}^{k +1} ) \\
		&\hspace{2.33cm}+(1 - \mu) \cdot \proj^{\{1,\ldots,k\}}(\bar{\lambda}^{k+1})
		\end{align*}
		with
		\begin{align*}
		\theta \cdot \mu  + \sum_{k' \in R^k} \theta_{k'} \cdot \mu + (1 - \mu) = \mu + (1 - \mu) = 1.
		\end{align*}	
	\end{enumerate}
This completes the induction and the proof.
\end{proof}\noindent

\end{appendices}

\end{document}